\documentclass[11pt,letterpaper]{article}
\usepackage{booktabs}
\usepackage[letterpaper,margin=0.98in]{geometry}

\usepackage{bm}
\usepackage{authblk}
\usepackage{mlmodern}
\usepackage{inconsolata}

\usepackage{mdframed}

\usepackage{dsfont}

\usepackage{amsfonts,nicefrac}
\usepackage[utf8]{inputenc}
\usepackage[dvipsnames]{xcolor}
\usepackage{graphicx}
\usepackage{multirow}
\usepackage{amsmath}
\usepackage{amsthm}
\usepackage{amssymb}
\usepackage{comment}
\usepackage{cancel}

\usepackage[shortlabels]{enumitem}
\usepackage[style=alphabetic, maxbibnames=15, giveninits=true, maxcitenames=15, natbib=true,
  maxalphanames=10, backend=biber, sorting=nty, backref=true, uniquename=false]{biblatex}
\DefineBibliographyStrings{english}{backrefpage = {page},backrefpages = {pages},}
\DeclareNameAlias{sortname}{given-family}
\renewbibmacro{in:}{\ifentrytype{article}{}{\printtext{\bibstring{in}\intitlepunct}}}
  
\makeatletter
\newtheorem{theorem}{Theorem}[section]

\newtheorem{lemma}[theorem]{Lemma}
\newtheorem{proposition}[theorem]{Proposition}
\newtheorem{definition}{Definition}[section]
\newtheorem{assumption}{Assumption}[section]
\theoremstyle{remark}
\newtheorem{remark}{Remark}[section]

\usepackage{hyperref}
\hypersetup{
  colorlinks   = true, urlcolor     = blue, linkcolor    = blue, citecolor   = red }
\makeatother

\renewcommand{\epsilon}{\varepsilon}

\usepackage{eqparbox}

\usepackage[capitalise]{cleveref}
\crefname{equation}{}{}

\title{Adaptive Matrix Online Learning through Smoothing with Guarantees for Nonsmooth Nonconvex Optimization \vspace{3mm}

\footnotetext{The authors are listed in alphabetical order.

$\! \! \quad$ The work of RJ was partially done while he was a Student Researcher at Google Research.}}
\usepackage{times}

\renewcommand*{\Affilfont}{\normalsize}

\renewcommand*{\Affilfont}{\normalsize}

\makeatletter
\renewcommand\AB@affilsepx{, \protect\Affilfont}
\renewcommand\AB@affilsep{\protect\Affilfont}
\makeatother

\author[$\dagger$]{Ruichen Jiang}
\author[$\ddagger$]{Zakaria Mhammedi}
\author[$\ddagger$]{Mehryar Mohri}
\author[$\dagger$,$\ddagger$]{Aryan Mokhtari}

\affil[$\dagger$]{UT Austin}
\affil[$\ddagger$]{Google Research}

\date{}

\usepackage{amsmath}
\usepackage{amssymb}
\usepackage{mathtools}
\usepackage{subcaption}
\usepackage{makecell,threeparttable}
\usepackage{algorithm,algorithmic}

\newcommand\Vector[1]{\bm{#1}}

\newcommand\ve{{\Vector{e}}}

\newcommand\vg{{\Vector{g}}}

\newcommand\vr{{\Vector{r}}}
\newcommand\vs{{\Vector{s}}}

\newcommand\vu{{\Vector{u}}}
\newcommand\vv{{\Vector{v}}}

\newcommand\vx{{\Vector{x}}}
\newcommand\vy{{\Vector{y}}}
\newcommand\vz{{\Vector{z}}}

\newcommand\vmu{{\Vector{\mu}}}

\newcommand\MATRIX[1]{\mathbf{#1}}

\newcommand\mA{{\MATRIX{A}}}
\newcommand\mB{{\MATRIX{B}}}

\newcommand\mD{{\MATRIX{D}}}
\newcommand\mE{{\MATRIX{E}}}

\newcommand\mG{{\MATRIX{G}}}
\newcommand\mH{{\MATRIX{H}}}
\newcommand\mI{{\MATRIX{I}}}
\newcommand\mJ{{\MATRIX{J}}}

\newcommand\mL{{\MATRIX{L}}}
\newcommand\mM{{\MATRIX{M}}}
\newcommand\mN{{\MATRIX{N}}}

\newcommand\mP{{\MATRIX{P}}}
\newcommand\mQ{{\MATRIX{Q}}}
\newcommand\mR{{\MATRIX{R}}}
\newcommand\mS{{\MATRIX{S}}}
\newcommand\mT{{\MATRIX{T}}}
\newcommand\mU{{\MATRIX{U}}}
\newcommand\mV{{\MATRIX{V}}}
\newcommand\mW{{\MATRIX{W}}}
\newcommand\mX{{\MATRIX{X}}}
\newcommand\mY{{\MATRIX{Y}}}
\newcommand\mZ{{\MATRIX{Z}}}

\newcommand\mSigma{{\bm{\Sigma}}}

\newcommand\sO{{\mathbb{O}}}

\newcommand\sR{{\mathbb{R}}}

\newcommand\bigO{\mathcal{O}}
\newcommand\semiS[1]{{\mathbb{S}_{+}^{#1}}}

\DeclareMathOperator*{\E}{\mathbb{E}}

\newcommand\op{{\mathrm{op}}}
\newcommand{\reals}{\mathbb{R}}

\DeclareMathOperator*{\argmax}{arg\,max}
\DeclareMathOperator*{\argmin}{arg\,min}

\DeclareMathOperator{\Tr}{Tr}

\DeclareMathOperator*{\tr}{Tr}

\newcommand{\cX}{\mathcal{X}}

\newcommand{\wt}{\widetilde}
\newcommand{\h}{\widehat}
\newcommand{\R}{\mathfrak R}

\newcommand{\reg}{\mathrm{Reg}}
\newcommand{\Reg}{\mathrm{Reg}}

\newcommand{\diag}{\mathrm{diag}}

\newcommand{\Breg}[3]{\mathcal{B}_{#1}\!\left(#2 \,\middle\|\, #3\right)}
\newcommand{\mzero}{\MATRIX{0}}
\newcommand{\dual}{_{\dagger}}
\newcommand{\mypara}[1]{\noindent\textbf{#1.}\ }

\begin{document}

\maketitle

\begin{abstract}We study online linear optimization with matrix variables constrained by the \emph{operator norm}, a setting where the geometry renders designing \emph{data-dependent} and \emph{efficient} adaptive algorithms challenging.
  The best-known adaptive regret bounds are achieved by Shampoo-like methods, but they require solving a costly quadratic projection subproblem. To address this, we extend the gradient-based prediction scheme to adaptive matrix online learning and cast algorithm design as constructing a family of smoothed potentials for the nuclear norm. We define a notion of admissibility for such smoothings and prove any admissible smoothing yields a regret bound matching the best-known guarantees of one-sided Shampoo. We instantiate this framework with two efficient methods that avoid quadratic projections. The first is an adaptive Follow-the-Perturbed-Leader (FTPL) method using Gaussian stochastic smoothing. The second is {Follow-the-Augmented-Matrix-Leader} (FAML), which uses a deterministic hyperbolic smoothing in an augmented matrix space. By analyzing the admissibility of these smoothings, we show both methods admit closed-form updates and match one-sided Shampoo's regret up to a constant factor, while significantly reducing computational cost. Lastly, using the online-to-nonconvex conversion, we derive two matrix-based optimizers, \emph{Pion} (from FTPL) and \emph{Leon} (from FAML). We prove convergence guarantees for these methods in nonsmooth nonconvex settings, a guarantee that the popular Muon optimizer lacks.
\end{abstract}

\newpage

\section{Introduction}

While Online Linear Optimization (OLO) is well-established for vector spaces in $\mathbb{R}^d$, modern applications such as neural network training are often natively cast in terms of matrix variables $\mX \in \mathbb{R}^{m \times n}$. In this context, relying on a naive reduction to the vector setting is problematic. Vectorization obscures intrinsic spectral structures, most notably operator-norm constraints, that are distinct from Euclidean geometry and essential for efficient optimization. To avoid the suboptimal regret bounds inherent in such reductions, we investigate OLO directly within the matrix domain. The problem is formally defined as follows:

\begin{mdframed}[skipabove=10pt,linewidth=1pt,innerleftmargin=.3em]\label{box:OL1}
    \textbf{Matrix Online Linear Optimization
    } \\For $t=1,\dots,T$:
    \begin{itemize}
            \vspace{-2mm}
        \item Learner chooses a matrix $\mX_t$ from a decision set $\mathcal{X}$;  
        \vspace{-2mm}
        \item Environment selects a gradient matrix $\mathbf{G}_t\in \mathbb{R}^{m \times n}$  defining the linear loss $\ell_t(\mX) = \langle \mathbf{G}_t, \mX \rangle$;
    \end{itemize}
                    \vspace{-2mm}
    \textbf{Goal:} Minimize 
$\Reg_T = \sum_{t=1}^T \langle \mathbf{G}_t, \mX_t \rangle - \min_{\mX \in \mathcal{X}} \sum_{t=1}^T \langle \mathbf{G}_t, \mX \rangle$.
\end{mdframed}

Matrix OLO has been extensively studied for problems like online variance minimization~\citep{warmuth2006online}, PCA~\citep{warmuth2008randomized}, and collaborative filtering~\citep{hazan2012near}, typically under nuclear-norm constraints. To address these, matrix multiplicative weight updates~\citep{arora2007combinatorial,arora2012multiplicative} have been proposed as a natural generalization of the classical Hedge algorithm. Matrix OLO with Schatten-$p$ constraints has also been applied to multi-task classification~\citep{agarwal2008matrix,cavallanti2010linear}, for which \cite{kakade2012regularization} developed mirror descent algorithms induced by matrix Bregman divergences.

In this paper, we focus on the case where the decision set $\mathcal{X}$ is an operator-norm ball, i.e., 
   $ \mathcal{X} = \{ \mX \in \mathbb{R}^{m \times n} : \|\mX\|_{\text{op}} \leq D \}$. As discussed in Section~\ref{subsec:motivating} {and detailed further in Appendix~\ref{app_motivating_example}}, this constraint naturally arises in classical problems such as learning rotations, as well as in modern optimization applications including the design of preconditioned gradient methods and neural network training. Our goal is to design adaptive online algorithms with data-dependent regret guarantees, in the spirit of AdaGrad~\citep{mcmahan2010adaptive,duchi2011adaptive}. In the matrix setting with operator-norm constraints, the strongest known regret guarantees are achieved by the one-sided Shampoo/ASGO methods proposed concurrently in~\citep{xie2025structured,an2025asgo}. Building on the Shampoo preconditioner~\citep{gupta2018shampoo}, these methods attain a regret bound of $\bigO(D\tr(\sqrt{\mM_T}))$, where $\mM_T = \sum_{t=1}^T \mG_t\mG_t^\top$, thereby offering better adaptivity over AdaGrad. However, enforcing feasibility requires solving a costly quadratic projection subproblem, which lacks a closed-form solution. This motivates the question: \emph{can one retain the same adaptive regret guarantees with greater computational efficiency?}

\medskip
\mypara{Contributions} We develop a unified framework for adaptive matrix online learning under operator-norm constraints that avoids quadratic projections while matching the best-known regret guarantees of one-sided Shampoo, answering the above question affirmatively. Our contributions are threefold.

First, we generalize the \textit{Gradient-Based Prediction Algorithm} (GBPA)~\citep{Abernethy2016} to adaptive matrix online algorithms, where algorithm design is cast as constructing a family of smoothed potentials parameterized by a positive semidefinite (PSD) matrix~$\mL$ that captures problem geometry. We formalize the notion of \emph{$(\alpha,\beta)$-admissible smoothings} of the \emph{nuclear norm} induced by the operator-norm constraint.
We show that, when the parameter $\mL$ is chosen adaptively, GBPA with any admissible potential family attains the regret bound $
\reg_T = \mathcal{O}\big(\sqrt{\alpha\beta}\,D\,\Tr(\sqrt{\mM_T})\big)$,
matching one-sided Shampoo up to constants (Theorem~\ref{thm:main}). Since the bound depends only on the {product} $\alpha\beta$, we characterize optimal admissibility by proving that any $(\alpha,\beta)$-admissible smoothing must satisfy $\alpha\beta \ge \tfrac{1}{2}$, and by exhibiting a regularized smoothing that attains this lower bound (Proposition~\ref{prop:smoothing}).

Second, to avoid the prohibitive cost of quadratic subproblems in one-sided Shampoo, we introduce two efficient online algorithms based on novel smoothed potentials. 
The first uses Gaussian stochastic smoothing, leading to an adaptive Follow-the-Perturbed-Leader (FTPL) algorithm that is parallelizable and relies on efficient matrix primitives;  
using noncentral Wishart theory, we show that this smoothing is admissible up to a mild dimension-dependent factor (Theorem~\ref{thm:regret_FTPL}).
Our main algorithm, {Follow-the-Augmented-Matrix-Leader} (FAML), is based on a \emph{deterministic} and \emph{explicit} hyperbolic smoothing tailored to the nuclear norm. By lifting to an augmented space, FAML admits closed-form updates, avoiding quadratic projections. It achieves near-optimal admissibility (Theorem~\ref{thm:FAML}) and matches one-sided Shampoo's regret up to a factor of two at substantially lower cost.\looseness=-1

Finally, we leverage our adaptive matrix online learning framework and apply the \textit{Online-to-Nonconvex Conversion} (O2NC) paradigm~\citep{cutkosky2023optimal} to obtain efficient matrix-based optimizers with provable guarantees for \emph{nonsmooth nonconvex} optimization. We formally identify the popular {Muon} optimizer as an instance of spectrally constrained FTL, clarifying its lack of guarantees in nonsmooth settings. In contrast, we introduce {Pion} (derived from FTPL) and {Leon} (derived from FAML), and establish that both converge to $(\rho,\varepsilon)$-stationary points for general nonsmooth nonconvex objectives (Theorems~\ref{thm_pion} and~\ref{thm:Leon}).

\subsection{Motivating Problems}\label{subsec:motivating}

Next, we focus on matrix-based deep learning optimization as the primary instance of the matrix OLO problem subject to operator-norm constraints. Additional applications, including learning rotations and online quasi-Newton updates, are deferred to Appendix~\ref{app_motivating_example}.

\vspace{1mm}

\textbf{Matrix Optimization Algorithms.} Deep learning architectures are inherently matrix-valued. Spectral optimizers like Muon \citep{jordan2024muon} exploit this structure, demonstrating significant empirical advantages over element-wise baselines. By leveraging the Online-to-Nonconvex Conversion (O2NC) framework \citep{cutkosky2023optimal,ahn2025general}, these methods can be rigorously modeled as Matrix OLO instances, where minimizing regret guarantees convergence to Goldstein stationary points. As detailed in Section~\ref{sec:app}, Muon corresponds precisely to spectrally-constrained FTL. We use this connection to derive a new optimization method with rigorous convergence guarantees in the nonsmooth setting, addressing a key theoretical gap in the standard Muon algorithm.\looseness=-1

\section{Adaptive Matrix Online Learning}\label{sec:adaptive_matrix_OL}

In the vector setting, it is known that achieving optimal regret requires aligning the adaptive preconditioner with the geometry of the constraint set; we defer a rigorous discussion of this alignment to Appendix~\ref{appen:vector_case}. Applying this principle to Matrix OLO via naive vectorization, however, encounters two fundamental barriers. First, treating an $m \times n$ matrix as a vector of dimension $d=mn$ implies a full preconditioner of size $d \times d$, incurring prohibitive storage and computational costs. Second, and more critically, vectorization obliterates the underlying spectral structure. Standard adaptive variants like Diagonal AdaGrad implicitly assume a hyper-rectangular constraint geometry, which is ill-suited for the spectral features of the operator-norm ball. Consequently, efficient optimization in this setting demands algorithms explicitly tailored to the matrix domain, which we review below.

A natural extension of adaptive Online Gradient Descent (OGD) to the matrix setting is characterized by the following update rule. Given a gradient matrix $\mG_t \in \mathbb{R}^{m \times n}$, we compute:
\begin{equation}\label{eq:shampoo}
    \mX_{t+1} = \argmin_{\mX \in \cX} \Bigl\{\langle \mG_t, \mX-\mX_t \rangle + \frac{1}{2\eta} \Tr((\mX-\mX_t)^\top\mL_t (\mX-\mX_t)\mR_t)  \Bigr\},
\end{equation}
where $\mL_t \in \mathbb{S}^{m \times m}_{++}$ and $\mR_t \in \mathbb{S}^{n \times n}_{++}$ denote the left and right preconditioners, respectively. This formulation admits a direct interpretation within the standard AdaGrad framework. Specifically, let $\vx = \text{vec}(\mX)$ and $\vg = \text{vec}(\mG)$ denote the vectorized decision variable and gradient obtained by stacking the columns of the matrix. 
The regularization term in \eqref{eq:shampoo} coincides with the quadratic form $\tfrac{1}{2\eta}(\vx-\vx_t)^\top \mH_t(\vx-\vx_t)$ from vector AdaGrad, with the additional constraint that the preconditioner factorizes as a Kronecker product $\mH_t=\mR_t\otimes \mL_t$. Under this structure, the penalty admits the efficient decomposition $(\vx - \vx_t)^\top (\mR_t \otimes \mL_t)(\vx - \vx_t) = \Tr\left( (\mX - \mX_t)^\top \mL_t (\mX - \mX_t)\mR_t \right)$.
Notably, Shampoo~\citep{gupta2018shampoo} and its one-sided variants~\citep{xie2025structured,an2025asgo} correspond to different choices of $\mL_t$ and $\mR_t$; we summarize these selections and the resulting regret bounds in Table~\ref{tab:shampoo} and note that the latter achieves the strongest known guarantee.

Despite the success of these methods in mitigating vectorization overhead and preserving matrix structure, they face a fundamental obstruction when constrained to the operator-norm ball, i.e., $    \mathcal{X} = \{ \mX \in \mathbb{R}^{m \times n} \colon \|\mX\|_{\text{op}} \leq D \}$. In this regime, the update step in \eqref{eq:shampoo} necessitates minimizing a convex quadratic objective over the operator-norm constraint, a problem that admits no closed-form solution. While practitioners often elide this projection step for computational expediency, such heuristics are known to invalidate worst-case regret guarantees \citep{orabona2018scale}.

Solving this subproblem at each round using iterative methods is also costly. Projection-based methods achieve linear convergence but necessitate a full Singular Value Decomposition (SVD) at every step. Conversely, Frank-Wolfe algorithms avoid SVDs by solving a linear subproblem, computable via an affordable polar factorization, but suffer from slower sublinear $\mathcal{O}(1/K)$ convergence. This requires a prohibitive number of oracle calls per update to achieve high precision. We detail complexity analysis in Appendix~\ref{appen:subproblem_complexity}. The core challenge is thus matching One-sided Shampoo's adaptive regret without prohibitive quadratic projections, which we address next.

\begin{table}[t]
    \centering
    \small
    \begin{threeparttable}
    \begin{tabular}{lll}
    \toprule
    \textbf{Algorithm}  & \textbf{Preconditioners} $\mL_t$ and $\mR_t$ & \textbf{Regret Bound} \\
    \midrule
    \makecell[l]{Shampoo$^\dagger$\\ \scriptsize\citep{gupta2018shampoo}}     & $\mL_t = \mM_t^{1/4},\mR_t = \mN_t^{1/4}$ & $ 
\|\cX\|_{\op} \max\limits_{\substack{
(a,b)\in\{(F,F),\\
(*,\op),(\op,*)\}
}}
  \|\mM_T^{1/4}\|_{a}\,\|\mN_T^{1/4}\|_{b}
    $ \\
    \makecell[l]{One-sided Shampoo\\ \scriptsize\citep{xie2025structured,an2025asgo}} &  {$\mL_t = \mM_t^{1/2},\mR_t = \mI$}   &   
    $\|\cX\|_{\op}\Tr(\mM_T^{1/2})$ 
\end{tabular}
    \end{threeparttable}
\caption{\small{{Here, $\mM_T = \sum_{t=1}^T \mG_t\mG_t^\top$ and $\mN_T = \sum_{t=1}^T \mG_t^\top\mG_t$. Moreover,  $\|\cX\|_{\op}= \max_{\mX, \mY \in \cX} \|\mX-\mY\|_{\op}$.\\ $^\dagger$~The reported regret for Shampoo is proven in Theorem~\ref{thm:improved_shampoo} (Appendix~\ref{appen:improved_shampoo}) which improves the original bound.  }}}
    \label{tab:shampoo}
\end{table}

\begin{remark}
While FTRL with Kronecker regularizers is expected to match OGD-style Shampoo guarantees, a corresponding analysis appears to be missing in prior work. Our framework supplies such Shampoo-type bounds for adaptive matrix FTRL; however, it still requires solving the same expensive operator-norm-constrained quadratic subproblems, motivating the more efficient approach developed later in the paper.
\end{remark}

\section{A Unified Framework for Adaptive Matrix Online Algorithms}\label{subsec:GBPA}

Before addressing the bottleneck of quadratic projections, we introduce a broad class of algorithms for Matrix Online Linear Optimization. We establish their regret guarantees via a unified framework that isolates the structure essential for spectral adaptivity. In Section~\ref{sec:algs}, we instantiate this framework to derive two concrete algorithms that effectively circumvent this computational barrier.

Our proposed framework can be considered as a generalization of the Gradient-Based Prediction Algorithm (GBPA) \citep{Abernethy2016} to the matrix domain. By incorporating potentials that capture intrinsic spectral geometry, we specialize this approach to operator-norm constraints, enabling near-optimal adaptive guarantees that are unattainable via standard vector reductions.

Recalling the Matrix OLO formulation, a GBPA strategy generates the action $\mX_{t+1}$ by evaluating the gradient of a convex potential $\wt{\Phi}_t : \mathbb{R}^{m \times n} \to \mathbb{R}$ at the cumulative gradient $\mS_t = \sum_{s=1}^t \mG_s$. To enforce the feasibility constraint $\|\mX_{t+1}\|_{\op} \le D$, we restrict the potentials such that $\|\nabla \wt{\Phi}_t\|_{\op} \le 1$ and scale the output by $D$. This yields the explicit update:
\vspace{0mm}
 \begin{equation}\label{eq:GBPA}
 \vspace{0mm}
      \mX_{t+1} = -D\nabla \wt{\Phi}_t(\mS_t) ,
 \end{equation}
for $t \geq 1$, and the initial action is set as $\mX_1=0$. 
Indeed, distinct choices of potential functions instantiate different algorithms for Matrix OLO. Now given the fact that our online learning problem is constrained by the operator-norm ball, we define the base potential function as $\Phi(\mS)=\max_{\|\mX\|_{\op}\leq 1}\langle \mX, \mS\rangle $ which can be further simplified as $\Phi(\mS)=\|\mS\|_*$.  Hence, the regret of the matrix OLO problem can be written as 
$
   \reg_T= \sum_{t=1}^T\langle \mG_t, \mX_t\rangle + D \Phi(\mS_T) =\sum_{t=1}^T\langle \mG_t, \mX_t\rangle + D \|\mS_T\|_* 
$.
The following lemma provides the central regret decomposition for this class of algorithms and is a direct corollary of \cite[Lemma 1.2]{Abernethy2016}; the proof is deferred to Appendix~\ref{appen:lem_GBPA}.\looseness=-1

\begin{lemma}\label{lem:GBPA}
     Define $\Breg{f}{\mU}{\mV}\!:= \!f(\mU) - f(\mV) - \langle \nabla f(\mV), \mU-\mV \rangle$ as the Bregman divergence with respect to a function $f$. If $\{\mX_t\}$ is generated by \eqref{eq:GBPA}, then its regret can be decomposed as:  
    \begin{equation*}
        \reg_T =  D\bigl(\Phi(\mS_T) - \wt{\Phi}_T(\mS_T)  \bigr) + DS\sum_{t=1}^{T-1} \Breg{\wt{\Phi}_{t}}{\mS_{t+1}}{\mS_{t}} + D\sum_{t=1}^{T-1} \bigl(\wt{\Phi}_{t+1}(\mS_{t+1}) - \wt{\Phi}_{t}(\mS_{t+1})\bigr) \!+D \wt{\Phi}_1(\mG_1).
    \end{equation*}
\end{lemma}
This regret decomposition admits three interpretable terms. The first term, $D\bigl(\Phi(\mS_T) - \wt{\Phi}_T(\mS_T)  \bigr)$, captures the underestimation error of the surrogate potential $\wt{\Phi}_T$ relative to the base potential $\Phi$. The second term involves the Bregman divergence induced by $\wt{\Phi}_t$ and reflects the smoothness of the surrogate potential along the trajectory $\{\mS_t\}$. The remaining accounts for the temporal variation of the potential sequence $\{\wt{\Phi}_t\}$ and measures its stability across rounds. 

To build intuition for choosing the surrogate potentials ${\wt{\Phi}_t}$, consider the simplest case where $\wt{\Phi}_t=\Phi$ for all $t\ge 1$. This choice recovers the classical Follow-the-Leader (FTL) algorithm, i.e., $
  \textstyle \mX_{t+1} = \argmin_{\|\mX\|_{\op} \leq D} \bigl\{ \langle \sum_{s=1}^t \mG_s, \mX \rangle \bigr\}$.
In this case, both the underestimation error and the temporal variation term in \eqref{eq:GBPA} vanish, yielding $\reg_T = D \sum_{t=1}^{T-1} \Breg{\Phi}{\mS_{t+1}}{\mS_{t}} + D \Phi(\mG_1)$. However, as the base potential $\Phi$, i.e., the nuclear norm, is nonsmooth, this can be exploited by the adversary to incur a large  Bregman divergence at each time step, leading to an $\Omega(T)$ regret bound. 

This motivates the construction of a sequence of surrogate potentials $\{\wt{\Phi}_t\}$ that closely approximates the base potential $\Phi$ while enjoying favorable smoothness properties. This is our main point of departure from \citep{Abernethy2016}. There, the authors select $\{\wt{\Phi}_t\}$ from a scalar-parametrized family $\{\wt{\Phi}_\eta: \eta > 0\}$, achieving an adaptive regret bound similar to Scalar AdaGrad (in our setting, $D\sqrt{\sum_{t=1}^T \|\mG_t\|_F^2}$). In contrast, to match the adaptive data-dependent guarantees of Shampoo (Table~\ref{tab:shampoo}), the smoothness of $\wt{\Phi}_t$ must depend on a preconditioner matrix, mirroring the adaptive regularizers in \eqref{eq:shampoo}. Accordingly, we choose $\{\wt{\Phi}_t\}$ from a family of potentials $\{\wt{\Psi}(\cdot;{\mL}): \mL \in \semiS{m}\}$ parametrized by a PSD matrix $\mL$, and formalize these desiderata in the following definition.

\begin{definition}\label{def:admissible_potential}
  Let $\{\wt{\Psi}(\cdot;{\mL}): \mL \in \semiS{m}\}$ be a family of potentials parametrized by a PSD matrix $\mL$. We say that $\wt{\Psi}(\cdot;\cdot)$ is an $(\alpha, \beta)$-admissible smoothing of $\|\cdot\|_*$ if the following conditions hold: 
  \begin{enumerate}[(a)]
    \item\label{item:feas} \emph{(Feasibility)} For all $\mL \in \semiS{m}$ and $\mX \in \reals^{m\times n}$, $\|\nabla_{\mX} \wt{\Psi}(\mX;\mL)\|_{\op} \leq 1$.
    \item\label{item:dominance} \emph{(Dominance)} For all $\mL \in \semiS{m}$ and $\mX \in \reals^{m\times n}$, $\wt{\Psi}(\mX;\mL) \geq \|\mX\|_*$ and $\tilde{\Psi}(\mX;\mzero) = \|\mX\|_*$.
    \item\label{item:stable} \emph{(Upper stability)} For any $\mL_1 \preceq \mL_2 $, $\sup_{\mX} (\wt{\Psi}(\mX;\mL_2) - \wt{\Psi}(\mX;\mL_1)) \leq \alpha (\Tr(\mL_2) - \Tr(\mL_1))$.
    \item\label{item:smooth} \emph{(Smoothness)} For any $\mL \succ 0$, $\wt{\Psi}(\cdot;\mL)$ is continuously differentiable and satisfies $\Breg{\wt{\Psi}(\cdot; \mL)}{\mY}{\mX} \leq \frac{\beta}{2}\tr((\mX-\mY)^\top \mL^{-1}(\mX - \mY))$. 
    
\end{enumerate}
\end{definition}
The conditions in Definition~\ref{def:admissible_potential} are directly motivated by the regret decomposition in Lemma~\ref{lem:GBPA} and are designed to control each term arising in \eqref{eq:GBPA}. 
Specifically, feasibility ensures that the iterates~$\mX_{t+1}$ produced by \eqref{eq:GBPA} satisfy the operator-norm constraint; Dominance guarantees that the underestimation term is always non-positive; Upper stability controls the temporal variation term; and smoothness bounds the Bregman divergence term.
Next, we show that $(\alpha,\beta)$-admissible potentials $\wt{\Psi}$ control the terms in Lemma~\ref{lem:GBPA}, yielding Shampoo-type regret for the presented GBPA class under a proper selection of $\mL_t$. See Appendix~\ref{appen:main} for the proof.

\begin{theorem}\label{thm:main}
Assume that $\|\mG_t\|_{\op} \le G$ for all $t \in [T]$, and let $\wt{\Psi}$ be an $(\alpha,\beta)$-admissible smoothing of the nuclear norm (Definition~\ref{def:admissible_potential}). Further, recall the definition $\mM_t = {\sum_{s=1}^t \mG_s \mG_s^\top}$. 
Consider the GBPA update~\eqref{eq:GBPA} with $\wt{\Phi}_t(\mS) = \wt{\Psi}(\mS;{\mL_t/\eta})$, where
\vspace{0mm}
\begin{equation}\label{eq:L_t}
\vspace{-1mm}
\mL_t \coloneqq \sqrt{G^2 \mI + \mM_t},\qquad \eta = \sqrt{\alpha/\beta}. 
\end{equation}
Then the regret of the algorithm satisfies 
\[\textstyle
\reg_T
\le
2\sqrt{\alpha\beta}\, D \, \Tr\Bigl(\sqrt{G^2\mI +  \mM_T}\Bigr)
+ (1-\sqrt{\alpha\beta})\,D\|\mG_1\|_* .
\]
\end{theorem}

A couple of remarks follows.
First, using the inequality $\Tr(\sqrt{G^2 \mI + \mM_T}) \leq mG + \Tr(\sqrt{\mM_T})$, the above bound simplifies to $\bigO(\sqrt{\alpha \beta}D(\tr(\sqrt{\mM_T})+mG))$. Neglecting the time-invariant term $mG$, this matches the one-sided Shampoo regret bound (Table~\ref{tab:shampoo}) up to the factor $\sqrt{\alpha \beta}$. Second, provided the parameterization $\mL = \mL_t/\eta$ is fixed, the choice of potential $\wt{\Phi}_t(\mS)$ affects the regret solely through the admissibility constants $\alpha$ and $\beta$ of its family.

Since the regret depends on the product $\alpha\beta$, two fundamental questions arise: what is the minimal achievable value for this product, and does there exist an $(\alpha, \beta)$-admissible potential family that achieves it? We resolve both in the following result (proof in Appendix~\ref{appen:smoothing}).

\begin{proposition}\label{prop:smoothing}
  For any $(\alpha, \beta)$-admissible smoothing $\tilde{\Psi}$, it holds that $\alpha \beta \geq \frac{1}{2}$. Moreover, the following smoothing $\wt{\Psi}^R$ is $(\frac{1}{2},1)$-admissible:
  \begin{equation}\label{eq:FTRL_potential}
    \wt{\Psi}^R(\mS; \mL) = \max_{\|\mX\|_{\op} \leq 1} \left\{\langle \mS, \mX\rangle - \frac{1}{2}\mathrm{Tr}(\mX^\top \mL \mX) \right\} + \frac{1}{2}\Tr(\mL).
  \end{equation}
\end{proposition}

By Danskin's theorem~\citep{bertsekas1999nonlinear}, the gradient of $\wt{\Psi}^R(\mS; \mL)$ is given by  $$\nabla_{\mS}\wt{\Psi}^R(\mS; \mL) = \argmax_{\|\mX\|_{\op} \leq 1} \left\{\langle \mS, \mX\rangle - \frac{1}{2}\mathrm{Tr}(\mX^\top \mL \mX) \right\}.$$ Consequently, by setting $\wt{\Phi}_t(\mS) = \wt{\Psi}^R(\mS; \mL_t/\eta)$ with $\mL_t$ defined in \eqref{eq:L_t} and choosing $\eta = 1/\sqrt{2}$ as prescribed by Theorem~\ref{thm:main}, the update induced by \eqref{eq:GBPA} (up to a sign change) can be written as 
\begin{equation}\label{eq:matrix_ftrl}
    \mX_{t+1} = D\argmin_{\|\mX\|_{\op} \leq 1} \Bigl\{ \langle \mS_t, \mX \rangle + \frac{1}{2\eta} \Tr(\mX^\top\mL_t \mX)  \Bigr\},
\end{equation} 
which corresponds to a Follow-the-regularized-Leader (FTRL) method. 
As a direct corollary, the update in \eqref{eq:matrix_ftrl} attains  the regret  $\sqrt{2}D (\tr(\sqrt{\mM_T})+mG)+(1-\frac{1}{\sqrt{2}})\|\mG_1\|_*$. Ignoring the lower-order terms $DmG$ and $\|\mG_1\|_*$, this matches the regret of one-sided Shampoo in Table~\ref{tab:shampoo} up to a constant factor. Notably, this provides the first such guarantee for an adaptive matrix FTRL method, as a side result of our framework.
However, like \eqref{eq:shampoo}, the update \eqref{eq:matrix_ftrl} necessitates solving a costly, iterative quadratic projection over the operator-norm ball. To address this, in the next section, we introduce two alternative potentials with substantially cheaper gradient evaluations.

\section{Proposed Algorithms}\label{sec:algs}
Next, we study alternative smoothings of the nuclear norm that yield more efficient algorithms. Section~\ref{subsec:FTPL} introduces a stochastic smoothing that leads to a Follow-the-Perturbed-Leader~(FTPL) method based on random perturbations, with parallelizable updates built from standard matrix primitives and an analysis relying on noncentral Wishart theory. Section~\ref{subsec:FAML} introduces  the novel {Follow-the-Augmented-Matrix-Leader}~(FAML) algorithm via a deterministic {hyperbolic smoothing}. We prove this smoothing attains near-optimal admissibility constants, matching those of regularized methods, while circumventing quadratic projections through efficient matrix primitives.

\subsection{Stochastic Smoothing: Follow-the-Perturbed-Leader}\label{subsec:FTPL}
As discussed, the central goal of GBPA is to construct a sequence of smooth, tractable surrogates for the nuclear norm that preserve its geometric properties while enabling efficient optimization. To this end, we employ \textit{stochastic smoothing}, a classical technique obtained by convoluting the nonsmooth objective with a smooth probability density function~\citep{glasserman1991gradient,yousefian2010convex,duchi2012randomized}, where the key design choice is the perturbation distribution. Intuitively, to achieve an adaptive regret bound comparable to that of one-sided Shampoo, the perturbations themselves should adapt to the previously observed gradient sequence. Drawing a parallel with the regularized potential $\wt{\Psi}^R$ in \eqref{eq:FTRL_potential}, we consider the following family of stochastic smoothing potentials:
\begin{equation}\label{eq:stochastic_smoothing}
         \wt{\Psi}^{S}(\mS; \mL) = \E_{\mZ \sim \mathcal{MN}(0,\mI_m, \mI_n)}\|\mS + \mL\mZ\|_*,
\end{equation}
where $\mathcal{MN}(0, \mI_m, \mI_n)$ is the matrix normal distribution with independent standard Gaussian entries.

To derive the gradient of $\wt{\Psi}^{S}(\mS; \mL)$, we use the variational representation of the nuclear norm and rewrite $\wt{\Psi}^{S}$ in \eqref{eq:stochastic_smoothing} as $\wt{\Psi}^{S}(\mS; \mL) = \E_{\mZ}\max_{\|\mX\|_{\op} \leq 1} \langle \mS + \mL \mZ, \mX\rangle$. By \citep[Proposition 2.2]{bertsekas1973stochastic}, we can swap the order of expectation and differentiation to obtain 
$
    \nabla\wt{\Psi}^{S}(\mS; \mL) = \E_{\mZ} [\argmax_{\|\mX\|_{\op} \leq 1} \langle \mS + \mL \mZ, \mX \rangle].
$
Following Theorem~\ref{thm:main}, we set $\wt{\Phi}_t(\mS) = \wt{\Psi}(\mS; \mL_t/\eta)$, where $\mL_t$ defined in \eqref{eq:L_t}. This choice leads to the following update:
\begin{equation}\label{eq:matrix_ftpl}
    \mX_{t+1} = -D \E_{\quad\mZ} \; \biggl[ \argmax_{\|\mX\|_{\op} \leq 1} \biggl\{  \langle \mS_t + \frac{1}{\eta}\mL_t\mZ, \mX\rangle\biggr\} \biggr], \quad \text{where } \mS_t := \sum_{s=1}^t \mG_s.
\end{equation}
The above update is indeed equivalent of an FTPL method with perturbation $(1/\eta) \mL_t \mZ$.

\textbf{Implementation and computational efficiency.} In contrast to the OGD update in \eqref{eq:shampoo} or the FTRL update in \eqref{eq:matrix_ftrl}, the update in \eqref{eq:matrix_ftpl} can be computed efficiently using standard linear algebra primitives such as Cholesky factorization and polar decomposition, as detailed below.  
To begin with, recall that $\mL_t = \sqrt{G^2 \mI + \mM_t}$, which may suggest explicitly computing a matrix square root. However, this is unnecessary; instead, we can compute the Cholesky factorization of $G^2 \mI + \mM_t = \wt{\mL}_t \wt{\mL}_t^\top$, and observe that $\wt{\mL}_t \mZ \stackrel{d}{=} \mL_t \mZ$ when $\mZ$ has independent standard Gaussian entries. The leading-order cost of Cholesky factorization is $\frac{1}{3}m^3$~\citep{golub2013matrix}.  

Moreover, for a fixed perturbation matrix $\mZ$, the maximization inside the expectation reduces to computing the \emph{polar factor} of $\mS_t + \frac{1}{\eta}\mL_t \mZ$. Specifically, for a  matrix $\mX \in \mathbb{R}^{m\times n}$ with singular value decomposition $\mX = \mU \mSigma \mV^\top$, its polar factor is  $\mathrm{polar}(\mX) = \mU \mV^\top$~\citep{higham2008functions}, which can be computed efficiently using well-established numerical iterative methods, including Newton-Schulz iteration~\citep{higham2008functions}, scaled Newton methods~\citep{higham2008functions}, and QDWH iterations~\citep{nakatsukasa2010optimizing,nakatsukasa2013stable}; see also~\cite{amsel2025polar}. In particular, Newton-Schulz iteration with $K$ steps incurs a computational cost of $4K m^2n$ (Appendix~\ref{appen:cost_FTPL}).

Finally, the update in \eqref{eq:matrix_ftpl} is presented in its \emph{deterministic} form based on the expected action. In practice, this expectation can be approximated via Monte Carlo sampling. Specifically, after computing the cumulative gradient $\mS_t = \sum_{s=1}^t \mG_t$ and the Cholesky factorization  $ \wt{\mL}_t \wt{\mL}_t^\top= G^2 \mI + \mM_t$ where $\mM_t = \sum_{s=1}^t \mG_t \mG_t^\top$, with $k$ samples of the random matrix $\mZ$ the update can be done as 
\begin{equation}\label{eq:matrix_ftpl_mc}
  \mX_{t+1}  = -\frac{D}{k}\sum_{i=1}^k  \mathrm{polar}\biggl(\mS_t + \frac{1}{\eta}\wt{\mL}_t\mZ_t^{(i)}\biggr), \quad \mZ_t^{(i)} \sim \mathcal{MN}(0, \mI_m, \mI_n) \quad \text{i.i.d.}
\end{equation}
Moreover, the $k$ polar factors in~\eqref{eq:matrix_ftpl_mc} can be computed in parallel. As a result, with sufficient parallel computing resources, the effective computational cost can be significantly reduced.

\textbf{Regret guarantees.} By Theorem~\ref{thm:main}, the only remaining task is to determine the admissibility parameters $(\alpha,\beta)$ for the stochastic smoothing in \eqref{eq:stochastic_smoothing}, which will lead to a regret guarantee on the FTPL algorithm in \eqref{eq:matrix_ftpl}. Using noncentral Wishart theory, we are able to establish the admissibility when $n \geq m+2$. The proof is presented in Appendix~\ref{appen:FTPL}.

\begin{theorem}\label{thm:regret_FTPL}
When $n \geq m+2$, the stochastic smoothing $\wt{\Psi}^S(\mS; \mL)$ is $(\alpha,\beta)$-admissible with $\alpha = \sqrt{m}+\sqrt{n}$ and $\beta =\frac{1}{\sqrt{n-m-1}} $. As a corollary, $\reg_T \leq 2\sqrt{2}D \left(\frac{n}{n-m-1} \right)^{1/4} \left(\Tr\left[\mM_T^{1/2}\right] + mG \right)$. 
\end{theorem}

Comparing with the regret bound of one-sided Shampoo in Table~\ref{tab:shampoo}, the bound in Theorem~\ref{thm:regret_FTPL} incurs an additional factor of $(\nicefrac{n}{(n-m-1)})^{1/4}$. That said, when $n \geq 2m$ and $m \geq 2$, this factor is at most $\sqrt{2}$ and the resulting guarantee matches one-sided Shampoo up to a constant factor.

\subsection{Hyperbolic Smoothing: Follow-the-Augmented-Matrix-Leader}\label{subsec:FAML}
While the regularized smoothing $\wt{\Psi}^{R}$ achieves optimal admissibility (Proposition~\ref{prop:smoothing}), its implicit definition via a quadratic program makes its gradient evaluation costly. In contrast, the randomized smoothing $\wt{\Psi}^{S}$ is cheaper to compute but incurs dimension-dependent admissibility factors (Theorem~\ref{thm:regret_FTPL}). This raises the question of whether one can obtain an explicit, inexpensive smoothing while retaining near-optimal admissibility.
Exploiting the special structure of the nuclear norm, we answer this question in the affirmative. Specifically, we introduce the hyperbolic family of potentials
\begin{equation}\label{eq:FAML_family}
    \wt{\Psi}^{H}(\mS;\mL)\coloneqq \tr\!\left(\sqrt{\mS \mS^\top + \mL \mL^\top}\right).
    \vspace{-1mm}
\end{equation}
Since $\|\mS\|_* = \tr(\sqrt{\mS\mS^\top})$, we have $\wt{\Psi}^{H}(\mS;\mzero)=\|\mS\|_*$, and the additive term $\mL\mL^\top$ provides an explicit smoothing. Similar smoothing of the nuclear norm such has been considered in in low-rank optimization, e.g.,~\citep{mohan2012iterative}. To the best of our knowledge, however, this particular smoothing has not been analyzed as a potential in an online learning framework.

To compute the gradient of $\tilde{\Psi}^H$, we exploit the interesting fact that the hyperbolic smoothing in \eqref{eq:FAML_family} can be interpreted as the nuclear norm of an augmented matrix. Specifically, define $\widehat{\mS} \coloneqq \begin{bmatrix}
        \mS & \mL
\end{bmatrix} \in \reals^{m \times (n+m)}$. Then $\widehat{\mS}\widehat{\mS}^\top=\mS\mS^\top+\mL\mL^\top$, and hence
$
    \wt{\Psi}^{H}(\mS;\mL) = \tr(\sqrt{\widehat{\mS}\widehat{\mS}^\top}) = \|\widehat{\mS}\|_*$. 
Consequently, the gradient with respect to $\mS$ can be obtained by first differentiating the nuclear norm in the augmented space and then restricting to the leading block. Concretely, letting $\widehat{\mX}= \nabla_{\h{\mS}}\|\h{\mS}\|_* = \argmax_{\|\mX\|_{\op} \leq 1} \langle \h{\mS}, \mX \rangle$, we obtain $\nabla \wt{\Psi}^{H}(\mS;\mL) = \widehat{\mX}[1\!:\!m, 1\!:\!n]$. With the choice of $\wt{\Phi}_t(\mS) = \wt{\Psi}^{H}(\mS; \mL_t/\eta)$ as in Theorem~\ref{thm:main}, where $\mL_t$ defined in \eqref{eq:L_t}, the update~\eqref{eq:GBPA} takes the form
\begin{align}\label{eq:FAML_update}
    \widehat{\mX}_{t+1} &= \arg\min_{\|\widehat{\mX}\|_{\text{op}} \le D} \langle \widehat{\mathbf{S}}_t, \widehat{\mX} \rangle
    \quad \text{where} \quad \widehat{\mathbf{S}}_t = \begin{bmatrix} \mathbf{S}_t & \frac{1}{\eta}\mathbf{L}_t \end{bmatrix} \in \mathbb{R}^{m \times (n+m)} \mspace{5mu} \text{and} \mspace{5mu}\nonumber
    \mathbf{S}_t = \sum_{s=1}^t \mathbf{G}_s \\
    \mX_{t+1} &= \widehat{\mX}_{t+1}[1:m, \, 1:n].
\end{align}
The above update rule can be interpreted as an FTL-style update applied to the augmented matrix~$\widehat{\mS}_t$, which is the reason we refer to it as the \emph{Follow-the-Augmented-Leader (FAML)} algorithm.

Alternatively, using the identity $\nabla_{\h{\mS}}\|\h{\mS}\|_* = (\h{\mS}\h{\mS}^\top)^{-1/2}\h{\mS}$, the gradient of \eqref{eq:FAML_family} can also be written as $\nabla_{\mS} \wt{\Psi}^{H}(\mS;\mL) = \left({\mS \mS^\top + \mL \mL^\top}\right)^{-1/2} \mS$. This leads to an equivalent formulation of FAML:
\begin{equation}\label{eq:FAML_update_equiv}
\mX_{t+1} = -\eta D\bigl({\eta^2\mS_t \mS_t^\top + (G^2 \mI + \mM_t)}\bigr)^{-1/2} \mS_t, \;  \text{where}\ \ \mS_t := \sum_{s=1}^t \mG_s,\, \mM_t := \sum_{s=1}^t \mG_s\mG_s^\top.
\end{equation}
It is instructive to compare~\eqref{eq:FAML_update_equiv} with the FTRL update~\eqref{eq:matrix_ftrl}. If we ignore the operator-norm constraint in \eqref{eq:matrix_ftrl}, the solution of the resulting unconstrained quadratic subproblem is $\mX_{t+1} = -\eta D \mL_t^{-1} \mS_t = -\eta D (G^2\mI + \mM_t)^{-1/2} \mS_t$.
In contrast, the FAML update in \eqref{eq:FAML_update_equiv} incorporates the additional term $\eta^2\mS_t \mS_t^\top$ inside the inverse square root. This guarantees that the iterate $\mX_{t+1}$ automatically satisfies the operator-norm constraint, without requiring an explicit projection. Consequently, FAML achieves essentially the same computational cost as solving an unconstrained FTRL subproblem, which explains how FAML alleviates the main computational bottleneck of matrix FTRL.   

\textbf{Implementation.} In contrast to the updates of OGD in \eqref{eq:shampoo} and FTRL in \eqref{eq:matrix_ftrl}, FAML can be implemented directly using \eqref{eq:FAML_update_equiv},  
where the dominant cost comes from computing a matrix inverse square root. 
While a standard approach is to use singular value decomposition, this can be costly. A more efficient alternative is to use iterative methods such as coupled Newton-Schulz~(NS) iteration~\citep{higham2008functions,an2025asgo}. As detailed in Appendix~\ref{appen:cost_FAML}, the leading-order cost in terms of floating-point operations is $6m^2n + 6 K m^3$, where $K$ denotes the number of inner NS steps.

A drawback of this approach is that computing matrix inverse square roots is numerically unstable, particularly in low precision. In this regard, the formulation in \eqref{eq:FAML_update} provides a more favorable alternative: similar to FTPL, it can be implemented by computing the polar factor of the augmented matrix $\widehat{\mS}_t$.
In fact, with a customized NS iteration tailored to this augmented formulation, we can avoid computing the matrix square root $\mL_t$ and the leading-order computational cost becomes $(2m^2n + 4m^3)K + 4m^2n$, where again $K$ is the number of inner NS steps (see Appendix~\ref{appen:cost_FAML}).

\textbf{Regret guarantees.} The last step is to determine the admissibility constants $\alpha$ and $\beta$ for the potential family defined in \eqref{eq:FAML_family} and, together with Theorem~\ref{thm:main}, to derive the corresponding regret bound for FAML in \eqref{eq:FAML_update}. The proof is presented in Appendix~\ref{appen:faml}.

\begin{theorem}\label{thm:FAML}
The hyperbolic smoothing $\wt{\Psi}^{H}(\mS;{\mL})$ defined in \eqref{eq:FAML_family} is $(\alpha,\beta)$-admissible with $\alpha = 1$ and $\beta = 1$. Hence, the regret of the FAML algorithm \eqref{eq:FAML_update} satisfies $\reg_T \leq 2D  (\Tr(\mM_T^{1/2}) + mG )$. 
\end{theorem}

Like FTRL \eqref{eq:matrix_ftrl}, this bound matches the one-sided Shampoo regret up to constant factors, neglecting the lower-order term $DmG$. Crucially, FAML eliminates the need to solve a quadratic program and is instead implemented using computationally efficient matrix primitives.

\section{Application: Nonsmooth Nonconvex Matrix Optimization}\label{sec:app}

We apply our adaptive online methods to nonsmooth nonconvex optimization using the \textit{Online-to-Nonconvex Conversion} (O2NC) framework. We start by reducing stochastic matrix optimization to Matrix OLO and identifying Muon~\citep{jordan2024muon} as an FTL instance. We then introduce \textit{Pion} and \textit{Leon}, theoretically grounded optimizers derived from our adaptive FTPL and FAML algorithms.

\subsection{Online-to-Nonconvex Conversion}\label{subsec:o2nc}

The O2NC framework, proposed by~\cite{cutkosky2023optimal} and refined in~\citep{zhang2024random,Ahn2024,ahn2025general}, reduces finding a stationary point of a stochastic matrix optimization problem to an online matrix optimization task.  To formalize this connection, we first define the stochastic matrix optimization problem:
\begin{equation}\label{eq:matrix_opt}\min_{\mW \in \mathbb{R}^{m\times n}}\; L(\mW) = \E_{\zeta \sim \mathcal{D}}[\ell(\mW; \zeta)],\end{equation}
where $L: \mathbb{R}^{m \times n} \rightarrow \mathbb{R}$ is differentiable. As $L$ may be nonsmooth and nonconvex, we adopt the notion of a $(\rho, \varepsilon)$-stationary point, a relaxation of the Goldstein stationary point~\citep{goldstein1977optimization}.

\begin{definition}[$(\rho, \varepsilon)$-stationary point]\label{def:stationary}
    Suppose $L$ is differentiable and let $\|\cdot\|$ be a norm with dual norm $\|\cdot\|_\dagger$. Then ${\mW}$ is a $(\rho,\varepsilon)$-stationary point if there exists a distribution $p$ over $\mathbb{R}^{m \times n}$ with $\E_{\mY \sim p}[\mY] = \mW$ such that 
    $
        \|\E[\nabla L(\mY)]\|_\dagger \leq \varepsilon$ and $\E\| \mY - \mW\| \leq \rho$.
\end{definition}

To simplify notations, for any $\mW \in \reals^{m \times n}$ and $\rho > 0$, define $\mathcal{P}(\mW; \rho)$ as the set of all distributions $p$ supported on $\reals^{m \times n}$ such that $\E_{\mY \sim p}[\mY] = \mW$ and $\E\| \mY - \mW\| \leq \rho$. Then ${\mW}$ is a $(\rho,\varepsilon)$-stationary point if and only if $\|\nabla L(\mW)\|_\dagger^{[\rho]} \coloneq \inf_{p \in \mathcal{P}(\mW; \rho)} \|\E_{\mY \sim p}[\nabla L(\mY)]\|_\dagger \leq \varepsilon$.

\paragraph{The Reduction Mechanism.} In the O2NC framework, for a given matrix norm $\|\cdot\|$, the optimizer queries an online learner at each step for a direction $\mX_{t+1} \in \{\mX \in \mathbb{R}^{m\times n}: \|\mX\| \leq D\}$ based on the history of observed gradients. The weights are then updated via $\mW_{t+1} = \mW_t + s_{t+1} \mX_{t+1}$ using randomized step sizes $s_{t+1} \sim \mathrm{Exp}(1)$.
As the gradients $\{\mG_t\}$ depend on evolving parameters $\mW_t$ through a nonconvex nonsmooth objective, the induced environment is highly non-stationary. To capture local geometry and decay stale information, we measure performance via \textit{discounted regret} with $\beta \in (0,1)$, defined as
$ \mathrm{Reg}_T^{[\beta]}(D) := \max_{\|\mX\| \leq D} [\sum_{t=1}^T \beta^{T-t} \langle \mG_t, \mX_t - \mX \rangle ]$.
Discounted regret can be reduced to standard regret by defining the loss at time $t$ as $\ell_t^{[\beta]}(\mX) = \langle\beta^{-t}\mG_t, \mX\rangle$ for all $t$ and then multiplying the resulting regret by $\beta^{T}$.
The following proposition establishes the key guarantee: if the learner achieves a discounted regret that is sublinear in $\frac{1}{1-\beta}$, the sequence converges to a $(\rho, \varepsilon)$-stationary point. Detailed protocols and proofs are provided in Appendix~\ref{appen:o2nc_formal}.

\begin{proposition}\label{prop:o2nc}
Let $\{\bar{\mW}_t\}$ denote the exponential moving average sequence and let $\tau$ be a random index defined in Appendix~\ref{appen:o2nc_formal}. If $D = \frac{(1-\beta)}{2\beta}\rho$, Then the expected $\rho$-stationarity gap at $\bar{\mW}_{\tau}$ satisfies
\begin{equation}\label{eq:bound_average_grad}
\E_{\tau}\!\left[\,
\|\nabla L(\bar{\mW}_{\tau})\|_\dagger^{[\rho]}
\right]
= \bigO\Bigl( \frac{L(\mW_0) - L(\mW^*)}{(1-\beta)\rho T} + {(1-\beta)}\E\left[\reg_{T}^{[\beta]}(1)\right]
+ \text{stochastic noise} \Bigr).
\end{equation}
\end{proposition}

Note that random sampling $\bar{\mW}_\tau$ is necessary as the last iterate lacks stationarity guarantees in nonconvex optimization. Since stochastic noise $\mE_t$ is algorithm-independent, minimizing the bound in \eqref{eq:bound_average_grad} reduces to minimizing discounted regret. We instantiate this framework with our adaptive FTPL and FAML updates, yielding \textit{Pion} and \textit{Leon}, to establish rigorous convergence guarantees.

\subsection{Matrix Optimization Algorithms: Muon, Pion, and Leon}\label{subsec:muon_pion}

In this section, we instantiate the O2NC framework with specific online learning algorithms. We first show that the Muon optimizer corresponds to the classic Follow-the-Leader (FTL) strategy, and then derive \textit{Pion} and \textit{Leon}, based on our Adaptive FTPL and FAML frameworks, respectively.

\textbf{Muon.}
Consider applying the FTL algorithm to the underlying online learning problem. At step $t$, the FTL update selects the matrix $\mX_{t+1}$ that minimizes the cumulative loss observed so far:
\begin{equation*}
    \mX_{t+1} = \argmin_{\|\mX\|_{\op} \leq D} \Bigl\{ \sum_{s=1}^t \ell_s^{[\beta]}(\mX) \Bigr\} = \argmin_{\|\mX\|_{\op} \leq D} \Bigl\{ \Bigl\langle \sum_{s=1}^t \beta^{-s}\mG_s , \mX \Bigr\rangle \Bigr\}=-D \cdot \mathrm{polar}\Bigl(\sum_{s=1}^t \beta^{-s} \mG_s\Bigr).
\end{equation*}
where the last equality follows from the fact that the solution to linear minimization over the operator-norm ball is given by the negative polar factor.

We now derive the optimization method for this update based on the reduction described in Section~\ref{subsec:o2nc}.
Since the polar decomposition is scale-invariant---satisfying $\mathrm{polar}(c\mA) = \mathrm{polar}(\mA)$ for any scalar $c > 0$---we may scale the argument by $\beta^t$ without affecting the result. By defining the exponential moving average (EMA) of the gradients as $\hat{\mG}_t = \sum_{s=1}^t \beta^{t-s} \mG_s$, the update becomes
\begin{equation*}
    \mW_{t+1} = \mW_t - s_{t+1} D \cdot \mathrm{polar}(\hat{\mG}_t), \quad s_{t+1}\sim \mathrm{Exp}(1).
\end{equation*}
Aside from the exponential scaling factor $s_{t+1}$, this derivation demonstrates that the Muon optimizer~\citep{jordan2024muon} is structurally equivalent to FTL under a spectral constraint. However, because FTL does not generally admit sublinear regret for nonsmooth losses, the O2NC framework cannot establish a rigorous convergence rate for Muon in such settings. To obtain these guarantees, we must adopt an algorithm with provable sublinear regret, which motivates our design of Pion.

\textbf{Pion.} We rectify the theoretical shortcomings of FTL by adopting the Adaptive FTPL framework, which stabilizes the optimization process through stochastic perturbation of the cumulative gradients. Extending \eqref{eq:matrix_ftpl_mc} to the discounted case and leveraging the scale invariance of the polar decomposition, the update direction of the resulted optimization algorithm would be 
\begin{equation}\label{eq:pion_update}
    \mX_{t+1} = -\frac{D}{k}\, \sum_{i=1}^k \biggl[ \mathrm{polar}\biggl({\sum_{s=1}^t \beta^{t-s}{\mG}_s} + {\frac{1}{\eta}\wt{\mL}_t\mZ_t^{(i)}}\biggr) \biggr],
\end{equation}
where $\mZ_t^{(i)} \overset{\text{i.i.d.}}{\sim} \mathcal{MN}(0,\mI_m,\mI_n)$ are standard Gaussian matrices, and the iterates are updated according to the standard O2NC protocol: $\mW_{t+1} = \mW_t + s_{t+1} \mX_{t+1}$, where $s_{t+1}\sim \mathrm{Exp}(1)$.

We name this algorithm \textbf{Pion}, as a perturbed variant of Muon. The perturbation is shaped by the discounted adaptive preconditioner $\wt{\mL}_t$, defined via the Cholesky factorization $
    \wt{\mL}_t \wt{\mL}_t^\top = {G^2 \beta^{-2}\mI + \hat{\mM}_t}$ with $\hat{\mM}_t = \sum_{s=1}^t (\beta^2)^{t-s}\mG_s \mG_s^\top$. This structured perturbation enables us to establish a valid regret bound and, consequently, convergence guarantees, properties that are not available for Muon.

We now establish the iteration complexity of Pion. We adopt standard assumptions on the stochastic gradients, consistent with prior work on adaptive methods~\citep{an2025asgo}.

\begin{assumption}[Bounded Stochastic Gradient]\label{assm:gradient_bound}
There exists a positive semidefinite matrix $\mQ \succeq 0$
    such that $\E[\nabla \ell(\mW_t;\zeta) \nabla \ell(\mW_t; \zeta)^\top] \preceq \mQ^2$ for all $t$.
\end{assumption}

Combining the regret guarantee of Adaptive FTPL (Theorem~\ref{thm:regret_FTPL}) with the O2NC reduction (Proposition~\ref{prop:o2nc}), we obtain the following non-asymptotic convergence rate.

\begin{theorem}[Convergence of Pion]\label{thm_pion}
For target accuracy $\varepsilon$ sufficiently small, Pion finds a $(\rho, \varepsilon)$-stationary point of Problem~\eqref{eq:matrix_opt} in $T$ iterations, where:
$T = \bigO\bigl(\max \bigl\{ \frac{C^2\|\mQ\|_*^2 \Delta L}{\rho \varepsilon^3},\frac{C^2 \|\mQ\|_*^2 }{\varepsilon^2}, \frac{ C m G}{\varepsilon}\bigr\}\bigr)$,
and $C = (n/(n-m-1))^{1/4}$ is the dimensional factor from the FTPL analysis.
\end{theorem}

\textbf{Leon.}
While Pion achieves robustness via stochastic perturbations, Monte-Carlo sampling can introduce additional variance. To address this, we introduce {Leon} (Learning-Enabled Orthogonalization and Normalization), a deterministic alternative based on our FAML algorithm.
The core idea of FAML is to replace the stochastic perturbation of FTPL with a deterministic augmentation in a higher-dimensional space. Instead of adding noise to the accumulated gradients, we construct an augmented matrix $\widehat{\mS}_t$ by concatenating the accumulated gradient with a scaled preconditioner.
Adapting the FAML update \eqref{eq:FAML_update} to the discounted O2NC setting, we define the augmented state using the gradient EMA $\hat{\mG}_t$ and the discounted preconditioner $\hat{\mL}_t$. The update direction is computed by projecting the polar factor of this augmented matrix back onto the original space:
\begin{equation}\label{eq:leon_derivation}
    \mX_{t+1} = \text{LeadingBlock}\left( -D \cdot \mathrm{polar}\left( \begin{bmatrix} \hat{\mG}_t & \frac{1}{\eta}\hat{\mL}_t \end{bmatrix} \right) \right).
\end{equation}
Leveraging the closed-form expression for the polar factor of a block matrix, this simplifies to:
\begin{equation}\label{eq:leon_update}
    \mX_{t+1} = -D \left( \hat{\mG}_t \hat{\mG}_t^\top + \frac{1}{\eta^2}(G^2 \beta^{-2}\mI + \hat{\mM}_t) \right)^{-1/2} \hat{\mG}_t,
\end{equation}
where $\hat{\mG}_t = \sum_{s=1}^t \beta^{t-s} \mG_s$ and $\hat{\mM}_t = \sum_{s=1}^t \beta^{2t-2s} \mG_s \mG_s^\top$. 
The final weight update follows the standard O2NC protocol: $\mW_{t+1} = \mW_t + s_{t+1} \mX_{t+1}$.
Leon can be interpreted as a preconditioned gradient method where the curvature correction is applied via a "smoothed" inverse square root. Unlike Muon, which lacks theoretical guarantees for nonsmooth objectives, Leon inherits the rigorous adaptive regret bounds of FAML without requiring random sampling.

\begin{theorem}[Convergence of Leon]\label{thm:Leon}
For target accuracy $\varepsilon$ sufficiently small, Leon finds a $(\rho, \varepsilon)$-stationary point of Problem~\eqref{eq:matrix_opt} in $T$ iterations, where:
  $
    T = \bigO\left(\max \left\{ \frac{\|\mQ\|_*^2 \Delta L}{\rho \varepsilon^3},\frac{\|\mQ\|_*^2 }{\varepsilon^2}, \frac{m G}{\varepsilon}\right\}\right)$.
\end{theorem}
Since FAML achieves tighter regret than FTPL, Leon improves Pion's convergence guarantee by a factor of $(n/(n-m-1))^{1/2}$ in dominant terms and $(n/(n-m-1))^{1/4}$ in the non-dominant term.

\subsection{Synthetic Empirical Validation}

\begin{figure}[t!]
\centering
    \small
    \includegraphics[width=0.92\linewidth]{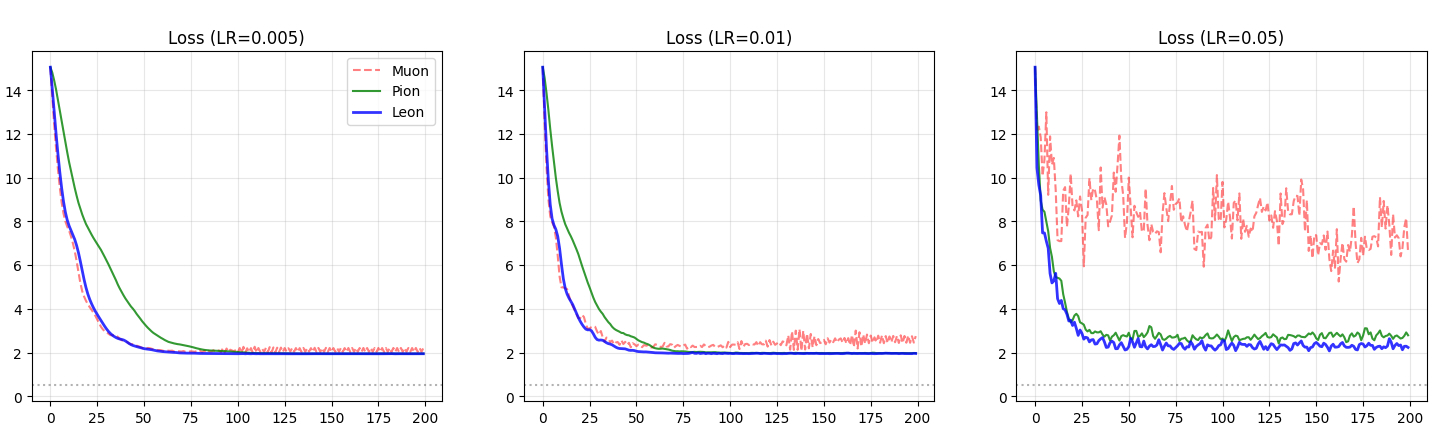}
\caption{Convergence paths with different  constant learning rates.}
    \vskip -.15in
    \label{fig:my_plot_2}
\end{figure}

To empirically validate our results, we compare \emph{Pion} and \emph{Leon} against \emph{Muon} on a synthetic Robust Matrix Sensing objective explicitly constructed to violate smooth-optimization assumptions. Defined by the superposition of an $\ell_1$-type absolute value term and high-frequency cosine ripples, this loss landscape generates abrupt gradient discontinuities designed to destabilize optimizers lacking intrinsic smoothing. We anticipate that \emph{Muon}---lacking such mechanisms---will react sharply to these fluctuations, whereas \emph{Pion} and \emph{Leon} will leverage the implicit smoothing of their Matrix OLO formulation to dampen discontinuities and ensure stable descent. The convergence paths in Figure~\ref{fig:my_plot_2} confirm this distinction: while Muon exhibits significant oscillation, both Pion and Leon demonstrate steady, monotonic convergence trajectories, with Leon yielding slightly superior rate performance consistent with our theoretical bounds. See Appendix~\ref{app:experiments} for more details. 

\section{Conclusion}

In this work, we addressed the computational challenges of adaptive matrix online learning constrained by the operator norm. By casting algorithm design as a smoothing problem, we developed two efficient methods, adaptive FTPL and FAML, that matched the best-known regret guarantees without solving costly quadratic projections. Finally, we extended this framework to nonsmooth nonconvex optimization via O2NC, introducing the Pion and Leon optimizers and establishing convergence guarantees for these methods that the popular matrix-based optimizer Muon lacks.

\newpage
\appendix

\makeatletter
\addtocontents{atoc}{} 

\let\orig@addcontentsline\addcontentsline
\renewcommand{\addcontentsline}[3]{\edef\@tempa{#1}\edef\@tempb{toc}\ifx\@tempa\@tempb
\orig@addcontentsline{atoc}{#2}{#3}\else
\orig@addcontentsline{#1}{#2}{#3}\fi
}
\makeatother

\clearpage
\section*{Contents of Appendix}
\setcounter{tocdepth}{3}
\makeatletter
\@starttoc{atoc}
\makeatother
\clearpage

\section{Additional Motivating Examples}\label{app_motivating_example}

\vspace{2mm}
\noindent \textbf{Learning rotations.} The problem of learning rotations aims to recover an underlying rotation matrix from a sequence of examples~\citep{arora2009learning,hazan2010learning,hazan2010corrigendum,nie2015optimal,hazan2016learning}. This fundamental problem arises in a wide range of applications, including computer vision, face recognition, robotics, crystallography, and physics; we refer the reader to~\cite{arora2009learning} for further references. Formally, at round $t$, the learner observes a unit vector $\vu_t \in \mathbb{R}^n$ and chooses an orthogonal matrix $\mX_t \in \sO(n)$. Subsequently, a target unit vector $\vv_t \in \mathbb{R}^n$ is revealed, and the Learner suffers the squared prediction error $\frac{1}{2}\|\vv_t - \mX_t \vu_t\|^2 = 1-\vv_t^\top \mX_t \vu_t$. 
As shown in \cite{hazan2010corrigendum,nie2015optimal}, one natural approach is to convexify the decision set $\sO(n)$ and use randomization to produce an orthogonal matrix. Since the convex hull of the orthogonal group is the operator-norm unit ball $\{\mX \in \sR^{n\times n}: \|\mX\|_{\op} \le 1\}$, this problem can be formulated as an instance of matrix OLO with an operator-norm constraint, where the loss function takes the form $\ell_t(\mX) = -\vv_t^\top \mX \vu_t = - \langle \vv_t \vu_t^\top, \mX\rangle$. 

\vspace{2mm}

\noindent \textbf{Online learning for quasi-Newton methods.}
Quasi-Newton methods accelerate optimization by maintaining a Hessian approximation, thereby avoiding the prohibitive cost of exact curvature computation. Recent advances \citep{jiang2023online,jiang2023accelerated,jiang2024online,jiang2025improved} have recast the update of this approximation as an online learning problem. Specifically, the optimizer sequentially selects a matrix $\mB_t$ from the spectral set $\mathcal{K} = \{\mB \in \mathbb{S}^n : \mu \mI \preceq \mB \preceq L\mI\}$. Subsequently, the environment reveals the iterate difference $\vs_t$ and gradient difference $\vy_t$, and the learner suffers the loss $\ell_t(\mB) = \frac{1}{2\|\vs_t\|^2} \|\vy_t - \mB \vs_t\|^2$. The regret bounds of this online sequence directly dictate the global convergence rate of the underlying optimization method~\citep{jiang2023online}.
The geometry of this problem is intrinsically spectral. The constraint $\mu \mI \preceq \mB_t \preceq L \mI$ serves a dual purpose: it guarantees a baseline linear convergence rate comparable to gradient descent and ensures the condition number of subproblems remains bounded. Crucially, this constraint set is affinely equivalent to the unit operator-norm ball constraint $\|\hat{\mB}\|_{\op} \leq 1$ via the transformation $\hat{\mB} = \frac{2}{L-\mu}(\mB - \frac{\mu+L}{2}\mI)$. Since the loss $\ell_t$ is convex, standard linearization reduces the task to minimizing a sequence of linear losses over this spectral set. Consequently, designing efficient quasi-Newton updates reduces to solving a Matrix OLO problem over the operator-norm ball.

\begin{table}[t]
    \centering
     \resizebox{\textwidth}{!}{
    \begin{tabular}{lll}
    \toprule
    \textbf{Algorithm}  & \textbf{Preconditioner} $\mH_t$ & \textbf{Regret Bound} \\
    \midrule
    \makecell[l]{Scalar AdaGrad-OGD\\ \scriptsize\citep{streeter2010less}}     & $\sqrt{\sum_{s=1}^t \|\vg_s\|_2^2} \mI$ & $ \|\cX\|_2\sqrt{\sum_{t=1}^T \|\vg_t\|_2^2}$ \\
    \makecell[l]{Diagonal AdaGrad-OGD\\\scriptsize\citep{streeter2010less,duchi2011adaptive}} &   $\sqrt{\diag(\sum_{s=1}^t \vg_s \vg_s^\top)}$   &   $\|\cX\|_\infty \tr\Bigl[\sqrt{\diag(\sum_{t=1}^T \vg_t \vg_t^\top)}\Bigr]$ \\
    \makecell[l]{Full-matrix AdaGrad-OGD\\\scriptsize\citep{duchi2011adaptive}} &   $\sqrt{\sum_{s=1}^t \vg_s \vg_s^\top}$    &  $\|\cX\|_2 \tr\Bigl[\sqrt{\sum_{t=1}^T \vg_t \vg_t^\top}\Bigr]$ \\
    \midrule
    \makecell[l]{Scalar AdaGrad-FTRL\\\scriptsize\citep{duchi2011adaptive}}     & $\delta\mI+\sqrt{\sum_{s=1}^t \|\vg_s\|_2^2} \mI$ & $ \|\cX\|_2\sqrt{\sum_{t=1}^T \|\vg_t\|_2^2} + \delta\|\cX\|_2$ \\
    \makecell[l]{Diagonal AdaGrad-FTRL\\\scriptsize\citep{duchi2011adaptive}} &   $\delta\mI+\sqrt{\diag(\sum_{s=1}^t \vg_s \vg_s^\top)}$   &   $\|\cX\|_\infty \sum_{i=1}^d\sqrt{\sum_{t=1}^T g_{t,i}^2}+ \delta\|\cX\|_1$ \\
    \makecell[l]{Full-matrix AdaGrad-FTRL\\\scriptsize\citep{duchi2011adaptive}} &   $\delta\mI + \sqrt{\sum_{s=1}^t \vg_s \vg_s^\top}$    &  $\|\cX\|_2 \tr\left[(\sum_{t=1}^T \vg_t \vg_t^\top)^{\frac{1}{2}}\right] + \delta \|\cX\|_2$
    \end{tabular}
    }
    \vskip -2mm
    \caption{Regret guarantees for AdaGrad for OGD and FTRL-type updates. Here, we define $\|\cX\|_p = \max_{\vx, \vy \in \cX} \|\vx\!-\!\vy\|_p$ for $p=1,2,\infty$. Moreover, $\delta \geq \max_{t}\|\vg_t\|_2$ for the scalar and full-matrix cases, and $\delta \geq \max_t\|\vg_t\|_\infty$ for the diagonal case.}
    \label{tab:OGD_adagrad}
\end{table}

\section{Adaptive Methods in the Vector Case}\label{appen:vector_case}

In this section, we review three canonical vector-based frameworks: (i) Online Gradient Descent (OGD), (ii) Follow-the-Regularized-Leader (FTRL), and (iii) Follow-the-Perturbed-Leader (FTPL). 
In the vector setting, these algorithms achieve adaptivity by explicitly tailoring the regularizer, local metric, or perturbation distribution to the structure of the constraint set. By clarifying how each framework exploits vector geometry, we motivate the need to adapt to spectral geometry in the matrix setting.

 A general formulation of \textbf{adaptive OGD} method ~\citep{streeter2010less,duchi2011adaptive} is given by the update:
\begin{equation}\label{eq:matrix_adagrad}
    \vx_{t+1} = \argmin_{\vx \in \cX} \Bigl\{\vg_t^\top \vx + \frac{1}{2\eta}(\vx - \vx_t)^\top \mH_t(\vx-\vx_t)\Bigr\},
\end{equation}
where $\vg_t$ is the gradient, $\eta>0$ is a scaling parameter, and $\mH_t$ is a preconditioner matrix constructed from the past gradients. 
The parameterization of $\mH_t$ as a scalar, diagonal, or full matrix yields distinct variants of adaptive OGD also known as AdaGrad, each characterized by specific regret guarantees (Table~\ref{tab:OGD_adagrad}). Consequently, there is no universally superior variant; the optimal choice of $\mH_t$ is strictly dictated by the geometry of the constraint set $\cX$. Achieving the tightest regret bounds requires aligning the structure of the preconditioner with the underlying geometry of the constraints.

\vspace{2em}\noindent\textbf{FTRL.} 
\cite{duchi2011adaptive} also proposed AdaGrad with FTRL update rules, which is given as follows:
\begin{equation*}
    \vx_{t+1} = \argmin_{\vx \in \cX} \Bigl\{\sum_{s=1}^t\vg_s^\top \vx + \frac{1}{2\eta}(\vx - \vx_1)^\top \mH_t(\vx-\vx_1)\Bigr\},
\end{equation*}
The preconditioner matrix $\mH_t$ is chosen similarly to the OGD case, except it requires an additional regularizer $\delta \mI$ due to the so-called ``off-by-one'' issue in the regret analysis~\citep{mcmahan2017survey}. The corresponding update rules and regret guarantees are presented in Table~\ref{tab:OGD_adagrad}. 

\begin{remark}
    Choosing the additional regularization term $\delta \mI$ properly requires knowledge of the maximum gradient norm, making the algorithm not fully adaptive. Specifically, the scalar and full-matrix cases require $\delta \geq \max_{t}\|\vg_t\|_2$, while the diagonal case requires $\delta \geq \|\vg_t\|_\infty$. For the scalar and diagonal cases, this can be fixed by either setting $\delta=0$ and using a more refined analysis~\citep{orabona2018scale}. Alternatively, one can use the clipping technique introduced in~\cite{cutkosky2019artificial}. Another solution is to use Proximal FTRL from \cite{mcmahan2017survey}.
\end{remark}

\vspace{2em}\noindent\textbf{FTPL.} Adaptive algorithms for FTPL are less explored. 
\cite{Abernethy2016} proposed an FTPL with Gaussian perturbation for online learning with $\ell_2$-Euclidean ball, given by 
\begin{equation*}
   \textstyle \vx_{t+1} = \E_{\vr \sim \mathcal{N}(0, \mI)}\left[\argmin_{\|\vx\|_2 \leq D} \left\{\left(\sum_{s=1}^t\vg_s + \eta_t\vr\right)^\top \vx\right\}\right].
\end{equation*}
It is shown that when $\eta_t$ is chosen adaptively, it recovers the same regret bound as Scalar AdaGrad-FTRL.

\subsection{Comparison}

We observe that the geometry of the decision set $\cX$ plays an important role in the regret of these AdaGrad variants, and it determines which variant is the best choice. In the following, we will consider two special cases: (i) $\cX$ is the $\ell_2$-norm ball; (ii) $\cX$ is the $\ell_\infty$-norm ball. Interestingly, both can be considered as a special case of the operator-norm ball in the matrix space. Specifically, when $m=1$, the matrix decision variable $\mX \in \mathbb{R}^{m\times n}$ can be identified by a vector $\vx \in \mathbb{R}^n$, and it holds that $\|\mX\|_{\op} = \|\vx\|_2$. Moreover, when $\mX$ is restricted to be diagonal, i.e., $\mX = \diag(\vx)$, then $\|\mX\|_{\op} = \|\vx\|_{\infty}$. 

\begin{remark}
    Since $\tr \bigl(\sqrt{\sum_{t=1}^T \vg_t \vg_t^\top}\bigr) \geq \sqrt{\tr \bigl(\sum_{t=1}^T \vg_t \vg_t^\top\bigr)} = \sqrt{\sum_{t=1}^T \|\vg_t\|_2^2}$, the regret bound of Full-matrix AdaGrad is always no better than that of Scalar AdaGrad; This phenomenon has been discussed in \cite{cutkosky2020better}. Thus, in the following, we focus on comparing the regret bound between Scalar AdaGrad and Diagonal AdaGrad. 
\end{remark}

\noindent\textbf{$\ell_2$-norm ball.} When $\cX = \{\vx \in \mathbb{R}^d: \|\vx\|_2 \leq D\}$, we have $\|\cX\|_2 = \|\cX\|_{\infty} = D$. By optimizing the scaling factor $\eta$, we obtain: 
\begin{equation*}
    \textstyle \text{Scalar AdaGrad: }\sqrt{2}D \sqrt{\sum_{t=1}^T \|\vg_t\|_2^2} \quad \text{vs.}\quad \text{Diagonal AdaGrad: }\sqrt{2}D  \sum_{i=1}^d \sqrt{\sum_{t=1}^T g_{t,i}^2}
\end{equation*}
In this case, one can show that the regret bound of Scalar AdaGrad is always no worse than that of Diagonal AdaGrad.

\noindent\textbf{$\ell_\infty$-norm ball.} When $\cX = \{\vx \in \mathbb{R}^d: \|\vx\|_\infty \leq D\}$, we have $\|\cX\|_2 = \sqrt{d}D$ and $\|\cX\|_{\infty} = D$. Similarly, after optimizing the choice of $\eta$, we obtain 
\begin{equation*}
    \textstyle \text{Scalar AdaGrad: }\sqrt{2d}D \sqrt{\sum_{t=1}^T \|\vg_t\|_2^2} \quad \text{vs.}\quad \text{Diagonal AdaGrad: }\sqrt{2}D  \sum_{i=1}^d \sqrt{\sum_{t=1}^T g_{t,i}^2}.
\end{equation*}
In this case, the ranking is reversed and Diagonal AdaGrad always outperforms Scalar AdaGrad. 

\section{Improved Analysis of Shampoo}\label{appen:improved_shampoo}
In the original Shampoo paper~\citep{gupta2018shampoo}, the authors establish the regret bound
\begin{equation}\label{eq:shampoo_original}
 \bigO\!\left( \|\cX\|_F \sqrt{r}\; \tr(\mM_T^{1/4})\, \tr(\mN_T^{1/4}) \right),
\end{equation}
where $\mM_T = \sum_{t=1}^T \mG_t\mG_t^\top$, $\mN_T = \sum_{t=1}^T \mG_t^\top\mG_t$,
$\|\cX\|_F = \max_{\mX,\mY\in\cX}\|\mX-\mY\|_F$, and $r$ is the maximum rank of $\{\mG_t\}_{t=1}^T$.
In this section, we show that a sharper regret bound for Shampoo can be obtained.

\begin{theorem}\label{thm:improved_shampoo}
The Shampoo algorithm in \eqref{eq:shampoo} with $\mL_t = \mM_t^{1/4}$ and $\mR_t = \mN_t^{1/4}$ achieves
\begin{equation}\label{eq:shampoo_regret}
\textstyle
\bigO\!\Bigl(
\|\cX\|_{\op}\,
\max\Bigl\{
\sqrt{\tr(\mM_T^{1/2})\,\tr(\mN_T^{1/2})},\;
\|\mM_T\|_{\op}^{1/4}\,\tr(\mN_T^{1/4}),\;
\tr(\mM_T^{1/4})\,\|\mN_T\|_{\op}^{1/4}
\Bigr\}
\Bigr).
\end{equation}
\end{theorem}

\begin{remark}
Using $\|\mM_T^{1/4}\|_F = \sqrt{\tr(\mM_T^{1/2})}$ and $\|\mM_T^{1/4}\|_{\op}=\|\mM_T\|_{\op}^{1/4}$,
the bound in \eqref{eq:shampoo_regret} admits the symmetric expression
\begin{equation*}
\textstyle
\bigO\!\Bigl(
\|\cX\|_{\op}\,
\max\Bigl\{
\|\mM_T^{1/4}\|_{F}\,\|\mN_T^{1/4}\|_{F},\;
\|\mM_T^{1/4}\|_{\op}\,\|\mN_T^{1/4}\|_{*},\;
\|\mM_T^{1/4}\|_{*}\,\|\mN_T^{1/4}\|_{\op}
\Bigr\}
\Bigr).
\end{equation*}
Compared to \eqref{eq:shampoo_original}, this bound removes the $\sqrt{r}$ factor and replaces the
constraint-set dependence $\|\cX\|_F$ by $\|\cX\|_{\op}$.
Moreover, since $\|\mA\|_F \le \|\mA\|_*$ and $\|\mA\|_{\op}\le \|\mA\|_*$, we have
\[
\max\Bigl\{
\|\mM_T^{1/4}\|_{F}\,\|\mN_T^{1/4}\|_{F},\;
\|\mM_T^{1/4}\|_{\op}\,\|\mN_T^{1/4}\|_{*},\;
\|\mM_T^{1/4}\|_{*}\,\|\mN_T^{1/4}\|_{\op}
\Bigr\}
\;\le\;
\|\mM_T^{1/4}\|_{*}\,\|\mN_T^{1/4}\|_{*}.
\]
In  regimes where $\mM_T^{1/4}$ and $\mN_T^{1/4}$ are close to full-rank, the improvement over
$\|\mM_T^{1/4}\|_{*}\|\mN_T^{1/4}\|_{*}$ can be as large as a factor on the order of $\min\{m,n\}$.
\end{remark}

The following proposition is a standard regret bound for Online Mirror Descent. 
\begin{proposition}\label{prop:regret_2}
Consider the update in \eqref{eq:shampoo}. Then, we have:
\begin{align*}\label{eq:prop:regret_2}
   &\phantom{{}\leq{}}\sum_{t=1}^T \langle \mG_t, \mX_t - \mX \rangle\\
    &\leq \frac{1}{2\eta} 
   \sum_{t=2}^T 
   \bigl[\tr \bigl((\mX_t-\mX)^\top\mL_t (\mX_t-\mX)\mR_t\bigr) - \tr \bigl((\mX_t-\mX)^\top\mL_{t-1} (\mX_t-\mX)\mR_{t-1}\bigr) \bigr] \\ 
   &\phantom{{}\leq{}}+\frac{1}{2\eta}\tr((\mX_1-\mX)^\top \mL_1 (\mX_1-\mX)\mR_1)+\frac{\eta}{2}\sum_{t=1}^T \tr \bigl( \mG_t^\top \mL_t^{-1} \mG_t \mR_t^{-1} \bigr).
\end{align*}
\end{proposition}

\begin{lemma}\label{lem:matrix_potential_appen}
  For any sequence of matrices $\{\mG_t\}_{t=1}^T$ with $\mG_t \in \mathbb{R}^{m \times n}$, define $\mM_t = \sum_{s=1}^t \mG_s \mG_s^\top$ and $\mN_t = \sum_{s=1}^t \mG_s^\top \mG_s$. Then it holds that
  \begin{equation*}
    \sum_{t=1}^T \tr(\mG_t^\top \mM_t^{-1/2}\mG_t) \leq 2\Tr(\sqrt{\mM_T}) \quad \text{and} \quad \sum_{t=1}^T \tr(\mG_t \mN_t^{-1/2}\mG_t^\top) \leq 2\Tr(\sqrt{\mN_T}). 
  \end{equation*}
\end{lemma}
\begin{proof}
  For ease of notation, define $\mM_0 = 0$ and $\mN_0 = 0$. Then we can write $\tr(\mG_t^\top \mM_t^{-1/2}\mG_t) = \tr( \mM_t^{-1/2}\mG_t \mG_t^\top) = \tr(\mM_t^{-1/2}(\mM_t - \mM_{t-1}))$ for any $t \geq 1$. Moreover, note that the matrix-valued function $\psi(\mM) = 2\Tr(\mM^{1/2})$ is concave and $\nabla \psi(\mM) = \mM^{-1/2}$. Thus, it holds that $\psi(\mM_{t-1}) -\psi(\mM_{t}) \leq  \langle \nabla\psi(\mM_t), \mM_{t-1} - \mM_{t} \rangle$, which is equivalent to $\Tr(\mM_t^{-1/2} (\mM_{t} - \mM_{t-1})) \leq 2(\Tr(\mM_t^{1/2}) - \Tr(\mM_{t-1}^{1/2}))$. Thus, we get  
    \begin{equation*}
       \sum_{t=1}^T \tr(\mG_t^\top \mM_t^{-1/2}\mG_t)\leq \sum_{t=1}^T 2(\Tr(\mM_t^{1/2}) - \Tr(\mM_{t-1}^{1/2}))  = 2\Tr(\mM_T^{1/2}). 
    \end{equation*}
    This proves the first inequality, and the second one follows from similar arguments. 
\end{proof}

\begin{lemma}\label{lem:first_term_shampoo}
  Recall that $\mM_t = \sum_{s=1}^t \mG_s \mG_s^\top$ and $\mN_t = \sum_{s=1}^t \mG_s^\top \mG_s$. With the choices $\mL_t = \mM_t^{1/4}$ and $\mR_t = \mN_t^{1/4}$, it holds that $\sum_{t=1}^T \tr \bigl( \mG_t^\top \mL_t^{-1} \mG_t \mR_t^{-1} \bigr) \leq 2\sqrt{\tr(\mM_T^{1/2})\tr(\mN_T^{1/2})}$. 
\end{lemma}
\begin{proof}
  By Cauchy-Schwarz inequality, we have $\tr \bigl( \mG_t^\top \mL_t^{-1} \mG_t \mR_t^{-1} \bigr) \leq \|\mL_t^{-1}\mG_t\|_F \|\mG_t \mR_t^{-1}\|_F$ for any $t \in \{1,\dots,T\}$. By applying Cauchy-Schwarz inequality again, we obtain that 
  \begin{equation}\label{eq:cs_shampoo}
    \sum_{t=1}^T \tr \bigl( \mG_t^\top \mL_t^{-1} \mG_t \mR_t^{-1} \bigr) \leq \sum_{t=1}^T \|\mL_t^{-1}\mG_t\|_F \|\mG_t \mR_t^{-1}\|_F \leq \sqrt{\sum_{t=1}^T \|\mL_t^{-1}\mG_t\|_F^2} \sqrt{\sum_{t=1}^T \|\mG_t \mR_t^{-1}\|_F^2}.
  \end{equation}
  Since $\mL_t = \mM_t^{1/4}$ and $\mR_t = \mN_t^{1/4}$, we have $\|\mL_t^{-1}\mG_t\|_F^2 = \tr\bigl(\mG_t^\top \mL_t^{-2} \mG_t\bigr) = \Tr(\mG_t^\top \mM_t^{-1/2} \mG_t)$ and similarly $\|\mG_t \mR_t^{-1}\|_F^2 = \Tr(\mG_t \mN_t^{-1/2} \mG_t^\top)$. The statement now follows by using Lemma~\ref{lem:matrix_potential_appen} in \eqref{eq:cs_shampoo}. 
\end{proof}

\begin{lemma}\label{lem:second_term_shampoo}
  We have 
  $\tr((\mX_1-\mX)^\top \mL_1 (\mX_1-\mX)\mR_1) + \sum_{t=2}^T 
   \bigl[\tr \bigl((\mX_t-\mX)^\top\mL_t (\mX_t-\mX)\mR_t\bigr) - \tr \bigl((\mX_t-\mX)^\top\mL_{t-1} (\mX_t-\mX)\mR_{t-1}\bigr) \bigr] \leq \|\cX\|_{\op}^2 \|\mM_T\|_{\op}^{1/4}\tr(\mN_T^{1/4}) + \|\cX\|_{\op}^2 \tr(\mM_T^{1/4})\|\mN_T\|_{\op}^{1/4}$. 
\end{lemma}
\begin{proof}
  We can write 
\begin{align*}
    & \phantom{{}={}}\langle \mX_t - \mX, \mL_t(\mX_t - \mX) \mR_t \rangle - \langle \mX_t - \mX, \mL_{t-1}(\mX_t - \mX) \mR_{t-1} \rangle \\
    &= \langle (\mX_t - \mX) \mR_t(\mX_t - \mX)^\top, \mL_t - \mL_{t-1} \rangle + \langle (\mX_t - \mX)^\top \mL_{t-1}(\mX_t - \mX), \mR_t - \mR_{t-1} \rangle. 
\end{align*}
Note that for two PSD matrices $\mA$ and $\mB$, $\langle \mA ,\mB \rangle \leq \|\mA\|_{\op} \tr(\mB)$. Therefore,  we have 
\begin{equation*}
    \langle (\mX_t - \mX) \mR_t(\mX_t - \mX)^\top, \mL_t - \mL_{t-1} \rangle \leq \|(\mX_t - \mX) \mR_t(\mX_t - \mX)^\top\|_{\op} \tr(\mL_t - \mL_{t-1}).
\end{equation*} 
Moreover, $\|\mA \mB\|_{\op} \leq \|\mA\|_{\op} \|\mB\|_{\op}$. Thus, we further have $$\|(\mX_t - \mX) \mR_t(\mX_t - \mX)^\top\|_{\op} \leq \|\mX_t - \mX\|^2_{\op} \|\mR_t\|_{\op} {{}\leq \|\cX\|_{\op}^2  \|\mR_T\|_{\op}}. $$ Combining these two together leads to $$\langle (\mX_t - \mX) \mR_t(\mX_t - \mX)^\top, \mL_t - \mL_{t-1} \rangle \leq \|\cX\|_{\op}^2{ \|\mR_T\|_{\op}}(\tr(\mL_t) - \tr(\mL_{t-1})).$$
Similarly, we have 
$$\langle (\mX_t - \mX)^\top \mL_{t-1}(\mX_t - \mX), \mR_t - \mR_{t-1} \rangle \leq {\|\cX\|_{\op}^2 \|\mL_{T}\|_{\op}}(\tr(\mR_t) - \tr(\mR_{t-1})).$$
Thus, by summing the inequality from $t=2$ to $t=T$, we conclude that 
\begin{align*}
& \phantom{{}\leq{}}\sum_{t=2}^T\left[ 
   \langle \mX_t - \mX, \mL_t(\mX_t - \mX) \mR_t \rangle 
   - \langle \mX_t - \mX, \mL_{t-1}(\mX_t - \mX) \mR_{t-1} \rangle
   \right] \\
   &\leq {\|\cX\|_{\op}^2 \|\mR_T\|_{\op}}(\tr(\mL_T) - \tr(\mL_{1})) + {\|\cX\|_{\op}^2 \|\mL_{T}\|_{\op}}(\tr(\mR_T) - \tr(\mR_{1})) \\
   & \leq \|\cX\|_{\op}^2 \|\mR_T\|_{\op}\tr(\mL_T) + \|\cX\|_{\op}^2\|\mL_{T}\|_{\op}\tr(\mR_T),
\end{align*}
where in the last inequality we used that $\tr(\mL_{1}) \geq 0$ and $\tr(\mR_{1}) \geq 0$. Finally, note that $\mL_T = \mM_T^{1/4}$ and $\mR_T = \mN_T^{1/4}$, and thus $\|\mL_T\|_{\op} = \|\mM_T^{1/4}\|_{\op} = \|\mM_T\|_{\op}^{1/4}$ and similarly $\|\mR_T\|_{\op}=\|\mN_T\|_{\op}^{1/4}$. This concludes the proof.  
\end{proof}
Applying Lemmas~\ref{lem:first_term_shampoo} and~\ref{lem:second_term_shampoo} to Proposition~\ref{prop:regret_2} leads to 
\begin{equation*}
  \sum_{t=1}^T \langle \mG_t, \mX_t - \mX \rangle \leq \frac{\|\cX\|_{\op}^2}{2\eta} \Bigl( \|\mM_T\|_{\op}^{1/4}\tr(\mN_T^{1/4}) +  \tr(\mM_T^{1/4})\|\mN_T\|_{\op}^{1/4} \Bigr) + \eta\sqrt{\tr(\mM_T^{1/2})\tr(\mN_T^{1/2})}.
\end{equation*}
By setting $\eta = \|\cX\|_{\op} / \sqrt{2}$, we obtain the bound in \eqref{eq:shampoo_regret}.

\section{Gradient-based Prediction Algorithm}

\subsection{Proof of Lemma~\ref{lem:GBPA}}\label{appen:lem_GBPA}
It follows from the definition of regret that  
\begin{equation}\label{eq:reg_decomp_1}
    \reg_T = \sum_{t=1}^T \langle {\mG}_t, \mX_t \rangle - \min_{\|\mX\|_{\op} \leq D} \sum_{t=1}^T \langle {\mG}_t, \mX \rangle = \sum_{t=1}^T \langle {\mG}_t, \mX_t \rangle + D\Bigl\|\sum_{t=1}^T {\mG}_t\Bigr\|_* .
\end{equation}
Recall that $\mS_t = \sum_{s=1}^t {\mG}_s$ and $\Phi(\mS) = \|\mS\|_*$, and hence the last term above is $D\Phi(\mS_T)$. Moreover, for $t\geq 1$, using the update in~\eqref{eq:GBPA}, we write 
\begin{align*}
    \langle {\mG}_{t+1}, \mX_{t+1} \rangle &= \langle \mS_{t+1} - \mS_{t},  -D\nabla\tilde{\Phi}_{t}(\mS_{t})\rangle \\
    &= D(\tilde{\Phi}_{t}(\mS_{t+1}) - \tilde{\Phi}_{t}(\mS_{t}) -\langle \nabla\tilde{\Phi}_{t}(\mS_{t}), \mS_{t+1} - \mS_{t} \rangle) - D\tilde{\Phi}_{t}(\mS_{t+1}) + D\tilde{\Phi}_{t}(\mS_{t}) \\
    &= D\Breg{\tilde{\Phi}_{t}}{\mS_{t+1}}{\mS_{t}} - D\tilde{\Phi}_{t}(\mS_{t+1}) +D\tilde{\Phi}_{t}(\mS_{t}). 
\end{align*}
Adding and subtracting $D\tilde{\Phi}_{t+1}(\mS_{t+1})$, this can be written as 
\begin{equation*}
  \langle {\mG}_{t+1}, \mX_{t+1} \rangle = D\Breg{\tilde{\Phi}_{t}}{\mS_{t+1}}{\mS_{t}}+ D(\tilde{\Phi}_{t+1}(\mS_{t+1}) - \tilde{\Phi}_{t}(\mS_{t+1})) + D(\tilde{\Phi}_{t}(\mS_{t}) - \tilde{\Phi}_{t+1}(\mS_{t+1})).
\end{equation*}
Summing the above inequality from $t=1$ to $t=T-1$, and noting that $\langle {\mG}_{1}, \mX_{1} \rangle = 0$, we obtain: 
\begin{equation}\label{eq:accum_loss}
    \sum_{t=1}^T \langle {\mG}_t, \mX_t \rangle  = D\sum_{t=1}^{T-1} \Breg{\tilde{\Phi}_{t}}{\mS_{t+1}}{\mS_{t}} + D\sum_{t=1}^{T-1} (\tilde{\Phi}_{t+1}(\mS_{t+1}) - \tilde{\Phi}_{t}(\mS_{t+1})) + D(\tilde{\Phi}_1(\mS_1) - \tilde{\Phi}_T(\mS_T)).
\end{equation}
Finally, combining \eqref{eq:reg_decomp_1} and \eqref{eq:accum_loss} yields the desired result. 

\subsection{Proof of Theorem~\ref{thm:main}}\label{appen:main}
We apply the regret decomposition in Lemma~\ref{lem:GBPA} and bound each term using properties from Definition~\ref{def:admissible_potential}.

\medskip
\mypara{Underestimation term}
By dominance (Definition~\ref{def:admissible_potential}, Property~\ref{item:dominance}), we have $$ \wt{\Phi}_T(\mS_T) = \wt{\Psi}(\mS_T; \mL_T/\eta) \geq \|\mS_T\|_* =  \Phi(\mS_T). $$
Hence, $
\Phi(\mS_T) - \wt{\Phi}_T(\mS_T) \le 0 .
$

\medskip
\mypara{Bregman divergence term} 
Since $\wt{\Phi}_t = \wt{\Psi}(\cdot; \mL_t/\eta)$, by smoothness (Property~\ref{item:smooth}) we have 
$$
\Breg{\wt{\Phi}_t}{\mS_{t+1}}{\mS_t}
\le
\frac{\beta\eta}{2}\,
\Tr((\mS_{t+1} - \mS_{t})^\top \mL_t^{-1} (\mS_{t+1} - \mS_{t})) = \frac{\beta\eta}{2}\tr(\mG_{t+1}^\top \mL_t^{-1} \mG_{t+1}).
$$
Recall that $\mM_t \coloneqq \sum_{s=1}^t \mG_s \mG_s^\top$. Since $\|\mG_{t+1}\|_{\op}\le G$, from~\eqref{eq:L_t} we have 
$
\mL_t = \sqrt{G^2\mI + \mM_t}
\succeq \sqrt{\mM_{t+1}} .
$
Applying Lemma~\ref{lem:matrix_potential_appen}, we obtain
$$
\sum_{t=1}^{T-1}
\Breg{\wt{\Phi}_t}{\mS_{t+1}}{\mS_t}
\le \frac{\beta\eta}{2} \sum_{t=1}^{T-1}\tr(\mG_{t+1}^\top \mL_t^{-1} \mG_{t+1}) \le
\beta\eta
\bigl(\Tr(\sqrt{\mM_T}) - \|\mG_1\|_*\bigr).$$

\medskip
\mypara{Stability term} Since $\mL_{t+1} \succeq \mL_t$ from \eqref{eq:L_t}, we can use upper stability (Property~\ref{item:stable}) to obtain 
\begin{equation*}
  \wt{\Phi}_{t+1}(\mS_{t+1}) - \wt{\Phi}_t(\mS_{t+1}) = \wt{\Psi}(\mS_{t+1}; \mL_{t+1}/\eta) - \wt{\Psi}(\mS_{t+1}; \mL_{t}/\eta) \leq \alpha\left(\frac{1}{\eta}\tr(\mL_{t+1}) - \frac{1}{\eta}\tr(\mL_t)\right).
\end{equation*}
Summing the above inequality from $t=1$ to $T-1$, we get
\[
\sum_{t=1}^{T-1}
\bigl(\wt{\Phi}_{t+1}(\mS_{t+1}) - \wt{\Phi}_t(\mS_{t+1})\bigr)
\le \sum_{t=1}^{T-1} \frac{\alpha}{\eta}(\Tr(\mL_{t+1}) - \Tr(\mL_t)) = 
\frac{\alpha}{\eta}\bigl(\Tr(\mL_T) - \Tr(\mL_1)\bigr).
\]
Moreover, using $\tilde{\Psi}(\mX;\mzero) = \|\mX\|_*$ and upper stability once more, we have 
$$
\tilde{\Phi}_1(\mG_1) = \tilde{\Psi}(\mG_1;\mzero) + \tilde{\Psi}(\mG_1;\mL_1/\eta) - \tilde{\Psi}(\mG_1;\mzero) \leq \|\mG_1\|_* + \frac{\alpha}{\eta}\tr(\mL_1).
$$
Combining the above,
$$
\sum_{t=1}^{T-1}
\bigl(\wt{\Phi}_{t+1}(\mS_{t+1}) - \wt{\Phi}_t(\mS_{t+1})\bigr)
+ \wt{\Phi}_1(\mG_1)
\le
\|\mG_1\|_* + \frac{\alpha}{\eta}\Tr(\mL_T).
$$

Finally, 
Using $\sqrt{\mM_T}\preceq \mL_T$ and collecting all terms,
$$
\reg_T
\le
D\Bigl(
(\beta\eta+\tfrac{\alpha}{\eta})\Tr(\mL_T)
+ (1-\beta\eta)\|\mG_1\|_*
\Bigr).$$
Choosing $\eta=\sqrt{\alpha/\beta}$ yields the stated bound.

\subsection{Proof of Proposition~\ref{prop:smoothing}}\label{appen:smoothing}
We first show the lower bound that $\alpha \beta \geq \frac{1}{2}$. Let $\wt{\Psi}$ be any $(\alpha,\beta)$-admissible smoothing of $\|\cdot\|_*$ according to Definition~\ref{def:admissible_potential}. 
By upper stability (Property~\ref{item:stable}), for any $\mL \succ 0$, 
$$\wt{\Psi}(\mzero; \mL) - \wt{\Psi}(\mzero; \mzero) \leq \alpha (\tr(\mL) - \tr(\mzero)) = \alpha \tr(\mL).$$ Moreover, from dominance in Property~\ref{item:dominance} we have $\wt{\Psi}(\mzero;\mzero) = \|\mzero\|_* = 0$, hence $\wt{\Psi}(\mzero; \mL) \le  \alpha \tr(\mL)$. Next, by smoothness (Property~\ref{item:smooth}), for any $\mX \in \mathbb{R}^{m \times n}$, it holds that 
   \begin{align*}
    \wt{\Psi}(\mX;\mL) &\leq \wt{\Psi}(\mzero;\mL) + \langle \nabla \wt{\Psi}(\mzero;\mL), \mX \rangle + \frac{\beta}{2}\tr(\mX^\top \mL^{-1} \mX), \\
    \wt{\Psi}(-\mX;\mL) &\leq \wt{\Psi}(\mzero;\mL) - \langle \nabla \wt{\Psi}(\mzero;\mL), \mX \rangle + \frac{\beta}{2}\tr(\mX^\top \mL^{-1} \mX).
   \end{align*}
   Averaging the two inequalities yields $\frac{1}{2}(\wt{\Psi}(\mX;\mL)+\wt{\Psi}(-\mX;\mL)) \leq \wt{\Psi}(\mzero;\mL) +  \frac{\beta}{2}\tr(\mX^\top \mL^{-1} \mX)$. Using dominance again, we have $ \wt{\Psi}(\pm \mX;\mL) \geq \|\mX\|_*$, and combining with the bound on $\wt{\Psi}(\mzero;\mL)$, we obtain: 
   \begin{equation*}
    \|\mX\|_* \leq \alpha \tr(\mL) +  \frac{\beta}{2}\tr(\mX^\top \mL^{-1} \mX), \quad \forall\; \mL \succ 0, \,\mX \in \mathbb{R}^{m \times n}.
   \end{equation*}
   Now fix $\varepsilon > 0$ and choose $\mL = \sqrt{\frac{\beta}{2\alpha}} \sqrt{\mX \mX^\top + \varepsilon \mI} \succ 0$. With this choice, the right-hand side becomes 
   $$
   \begin{aligned}
    \alpha \tr(\mL) +  \frac{\beta}{2}\tr(\mX^\top \mL^{-1} \mX) &= \sqrt{\frac{\alpha \beta}{2}} \tr(\sqrt{\mX \mX^\top + \varepsilon \mI}) +  \sqrt{\frac{\alpha \beta}{2}} \tr(\mX^\top (\mX \mX^\top + \varepsilon \mI)^{-1/2} \mX) \\
    &\leq \sqrt{2{\alpha \beta}} \tr(\sqrt{\mX \mX^\top + \varepsilon \mI}).
   \end{aligned}
   $$
   Letting $\varepsilon \rightarrow 0$ and using $\tr(\sqrt{\mX\mX^\top}) =\|\mX\|_*$, we conclude that $\|\mX\|_* \leq \sqrt{2 \alpha \beta} \|\mX\|_*$ for all $\mX$. This is possible only if $\alpha\beta \geq \frac{1}{2}$, completing the proof of the lower bound.  

   In the remaining, we show that the regularized smoothing defined in \eqref{eq:FTRL_potential} is $(\frac{1}{2},1)$-admissible and verify the properties in Definition~\ref{def:admissible_potential} one by one.

   \medskip
  \mypara{Feasibility}
   By Danskin's theorem~\citep{bertsekas1999nonlinear}, the gradient of $\wt{\Psi}^R(\mS; \mL)$ is given by solving the maximization problem 
   \begin{equation}\label{eq:danskin}
       \nabla_{\mS}\wt{\Psi}^R(\mS; \mL) = \argmax_{\|\mX\|_{\op} \leq 1} \left\{\langle \mS, \mX\rangle - \frac{1}{2}\mathrm{Tr}(\mX^\top \mL \mX) \right\}.
   \end{equation}
   In particular, this implies that $\nabla_{\mS}\wt{\Psi}^R(\mS; \mL)$ is a feasible point and thus $\|\nabla_{\mS}\wt{\Psi}^R(\mS; \mL)\|_{\op} \leq 1$. 

  \medskip
\mypara{Dominance} It is easy to verify that $\wt{\Psi}^R(\mS; \mzero) = \max_{\|\mX\|_{\op} \leq 1} \langle \mS, \mX \rangle = \|\mS\|_*$. Moreover, for any fixed $\mX$ that satisfies $\|\mX\|_{\op} \leq 1$ and $\mL \in \semiS{m}$, it holds that $\Tr(\mX^\top \mL \mX) \leq \Tr(\mL) \|\mX\|_{\op}^2 \leq \Tr(\mL)$.  Hence, we further have 
\begin{equation*}
  \langle \mS, \mX\rangle - \frac{1}{2}\mathrm{Tr}(\mX^\top \mL \mX)  + \frac{1}{2}\Tr(\mL) \geq \langle \mS, \mX\rangle.
\end{equation*}
Maximizing both sides over $\mX \in \{\mX: \|\mX\|_{\op} \leq 1\}$, we recognize the left-hand side as $\wt{\Psi}^R(\mS; \mL)$, while the right-hand side yields $\max_{\|\mX\|_{\op} \leq 1}\langle \mS, \mX\rangle = \|\mS\|_*$. Hence, this proves that $\wt{\Psi}^R(\mS; \mL) \geq \|\mS\|_*$ for all $\mL \in \semiS{m}$ and $\mX$.  

\medskip
\mypara{Upper stability} Consider any $\mL_1 \preceq \mL_2$ and fix $\mS \in \reals^{m \times n}$. For any $\mX \in \reals^{m \times n}$, we have $\tr(\mX^\top \mL_1 \mX) \leq \tr(\mX^\top \mL_2 \mX)$. This further implies that
\begin{equation*}
  \Bigl\{\langle \mS, \mX\rangle - \frac{\mathrm{Tr}(\mX^\top \mL_2 \mX)}{2} + \frac{\Tr(\mL_2)}{2} \Bigr\} \leq \Bigl\{\langle \mS, \mX\rangle - \frac{\mathrm{Tr}(\mX^\top \mL_1 \mX)}{2} + \frac{\Tr(\mL_1)}{2}\Bigr\} + \frac{\tr(\mL_2) - \tr(\mL_1)}{2}. 
\end{equation*} 
Maximizing both sides over $\mX \in \{\mX: \|\mX\|_{\op} \leq 1\}$ yields $$\wt{\Psi}^R(\mS; \mL_2) \leq \wt{\Psi}^R(\mS; \mL_1) + \frac{1}{2}(\tr(\mL_2) - \tr(\mL_1)).$$ Since $\mS$ is arbitrary, this proves Property~\ref{item:stable} holds with $\alpha = \frac{1}{2}$. 

\medskip
\mypara{Smoothness}
Fix $\mL \succ 0$ and let $\iota_{\|\mX\|_{\op}\le 1}(\mX)$ denote the indicator function of the operator-norm unit ball.
Define
\[
g(\mX)
\;\coloneqq\;
\iota_{\|\mX\|_{\op}\le 1}(\mX)
\;+\;
\frac12 \tr(\mX^\top \mL \mX).
\]
Then the regularized potential in~\eqref{eq:FTRL_potential} can be written as
\[
\wt{\Psi}^R(\mS;\mL)
=
\max_{\mX}\bigl\{\langle \mS,\mX\rangle - g(\mX)\bigr\}
\;+\;
\frac12 \Tr(\mL),
\]
where $\langle \mS,\mX\rangle=\tr(\mS^\top\mX)$ denotes the Frobenius inner product.
Ignoring the additive constant $\frac12\Tr(\mL)$, which does not depend on $\mS$, we may view
$\wt{\Psi}^R(\cdot;\mL)$ as the Fenchel conjugate of $g$.

Since $\frac12\tr(\mX^\top \mL \mX)$ is $1$-strongly convex with respect to the norm
\[
\|\mX\|_{\mL} \coloneqq \sqrt{\tr(\mX^\top \mL \mX)},
\]
and the indicator $\iota_{\|\mX\|_{\op}\le 1}$ is convex, the function $g$ is $1$-strongly convex with respect to $\|\cdot\|_{\mL}$.
By Fenchel duality, its conjugate $\wt{\Psi}^R(\cdot;\mL)$ is therefore $1$-smooth with respect to the dual norm
\[
\|\mS\|_{\mL^{-1}} \coloneqq \sqrt{\tr(\mS^\top \mL^{-1} \mS)}.
\]
Consequently, for any $\mS_1,\mS_2$, the associated Bregman divergence satisfies
\[
\begin{aligned}
\Breg{\wt{\Psi}^R(\cdot;\mL)}{\mS_2}{\mS_1}
&=
\wt{\Psi}^R(\mS_2;\mL)
-
\wt{\Psi}^R(\mS_1;\mL)
-
\langle \nabla_{\mS}\wt{\Psi}^R(\mS_1;\mL),\mS_2-\mS_1\rangle \\
&\le
\frac12
\tr\bigl((\mS_2-\mS_1)^\top \mL^{-1}(\mS_2-\mS_1)\bigr).
\end{aligned}
\]
This verifies Property~\ref{item:smooth} with $\beta=1$.

\section{Proofs for Section~\ref{sec:algs}}

\subsection{Proofs for FTPL (Theorem~\ref{thm:regret_FTPL})}\label{appen:FTPL}
In this section, we verify that the stochastic smoothing potential family $\wt{\Psi}^{S}(\mS; \mL)$ satisfies all the properties in Definition~\ref{def:admissible_potential} with $(\alpha,\beta) = (\sqrt{m}+\sqrt{n}, \frac{1}{\sqrt{n-m-1}})$, and the regret bound directly follows from Theorem~\ref{thm:main}. The first three are relatively straightforward while characterizing the smoothness property is the main challenge here; as we shall see, it requires tools from noncentral Wishart theory. 

\medskip
\mypara{Feasibility} Using the variational representation of the nuclear norm $\|\mS\|_* = \max_{\|\mX\|_{\op} \leq 1} \langle \mS, \mX\rangle$ and exchanging the order of expectation and differentiation via \citep[Proposition 2.2]{bertsekas1973stochastic}, we have
$$
    \nabla\wt{\Psi}^{S}(\mS; \mL) = \E_{\mZ} \Bigl[\argmax_{\|\mX\|_{\op} \leq 1} \langle \mS + \mL \mZ, \mX \rangle\Bigr].
$$
Thus, since the operator norm is a convex function, by Jensen's inequality we have 
$$\|\nabla\wt{\Psi}^{S}(\mS; \mL)\|_{\op} \leq  \E_{\mZ} \Bigl[\Bigl\|\argmax_{\|\mX\|_{\op} \leq 1} \langle \mS + \mL \mZ, \mX \rangle \Bigr\|_{\op}\Bigr] \leq 1.$$

\medskip
\mypara{Dominance} It is easy to see that when $\mL = 0$, the stochastic perturbation vanishes and hence $\wt{\Psi}^{S}(\mS; \mzero) = \|\mS\|_*$. Moreover, since the nuclear norm is convex and $\E[\mZ] = \mzero$, Jensen's inequality implies that
\begin{equation*}
    \wt{\Psi}^{S}(\mS; \mL) = \E_{\mZ \sim \mathcal{MN}(0,\mI_m, \mI_n)}\|\mS + \mL\mZ\|_* \geq   \bigl\|\mS + \mL\E_{\mZ \sim \mathcal{MN}(0,\mI_m, \mI_n)}[\mZ]\bigr\|_* = \|\mS\|_*.
\end{equation*}
This proves Property~\ref{item:dominance}. 

\medskip
\mypara{Upper stability} 
Consider any two matrices $\mL_1 \preceq \mL_2$ and fix $\mS \in \reals^{m \times n}$. Using the linearity of expectation and triangle inequality, we have 
\begin{align*}
  \tilde{\Psi}^S(\mS; \mL_2)- \tilde{\Psi}^S(\mS; \mL_1) &= \E_{\mZ \sim \mathcal{MN}(0,\mI_m, \mI_n)}\|\mS + \mL_2\mZ\|_* - \E_{\mZ \sim \mathcal{MN}(0,\mI_m, \mI_n)}\|\mS + \mL_1\mZ\|_* \\
  & = \E_{\mZ \sim \mathcal{MN}(0,\mI_m, \mI_n)} \left[\|\mS + \mL_2\mZ\|_* - \|\mS + \mL_1\mZ\|_* \right] \\
  & \leq \E_{\mZ \sim \mathcal{MN}(0,\mI_m, \mI_n)} \left[\|(\mL_2-\mL_1)\mZ\|_* \right].
\end{align*}
Moreover, it holds that $\|(\mL_2-\mL_1)\mZ\|_* \leq \|\mZ\|_{\op}\|\mL_2-\mL_1\|_*$, and Gordon's inequality for Gaussian matrices~\citep[Theorem 5.32]{vershynin2010introduction} establishes that $\E_{\mZ \sim \mathcal{MN}(0,\mI_m,\mI_n)}\|\mZ\|_{\op} \leq \sqrt{m} + \sqrt{n}$. 
Using the fact that $\|\mL_2-\mL_1\|_* = \Tr(\mL_2 - \mL_1) = \Tr(\mL_2) - \Tr(\mL_1)$ since $\mL_2 - \mL_1 \succeq \mzero$, we obtain 
\begin{equation*}
  \tilde{\Psi}^S(\mS; \mL_2)- \tilde{\Psi}^S(\mS; \mL_1) \leq \E_{\mZ \sim \mathcal{MN}(0,\mI_m, \mI_n)} \|\mZ\|_{\op} \|\mL_2-\mL_1\|_* \leq (\sqrt{m} + \sqrt{n})(\Tr(\mL_2) - \Tr(\mL_1)).
\end{equation*}
This proves that Property~\ref{item:stable} is satisfied with $\alpha = \sqrt{m} + \sqrt{n}$.

\medskip
\mypara{Smoothness} To simplify the notation, in the following, we view the Hessian of a matrix function $F : \reals^{m \times n} \rightarrow \reals$ as a bilinear map on $\reals^{m \times n}$. Equivalently, under the Fronbenius inner product, we identify it with a self-adjoint linear operator and write $$\nabla^2 F(\mX)[\mD_1, \mD_2] = \langle \nabla^2 F(\mX)[\mD_1], \mD_2\rangle, \quad \forall\, \mD_1, \mD_2 \in \reals^{m \times n}. $$

Fix $\mL \succ 0$. First, since the perturbation follows a matrix Gaussian distribution, by \citep[Lemma 1.5]{Abernethy2016}, $\wt{\Psi}^S(\mS; \mL)$ is twice differentiable. For any $\mS_1, \mS_2 \in \reals^{m \times n}$, let $\Delta \mS = \mS_2 - \mS_1$. By the fundamental theorem of calculus, 
\begin{equation}\label{eq:FTC1}
  \wt{\Psi}^S(\mS_2; \mL) - \wt{\Psi}^S(\mS_1; \mL) = \int_{0}^1 \langle \nabla_{\mS} \wt{\Psi}^S(\mS_1+ t \Delta \mS; \mL), \Delta \mS \rangle \,dt.
\end{equation} 
For any fixed $t \in [0,1]$, we apply the fundamental theorem of calculus again to obtain 
\begin{align}
  \nabla_{\mS} \wt{\Psi}^S(\mS_1+ t \Delta \mS; \mL) - \nabla_{\mS} \wt{\Psi}^S(\mS_1; \mL) &= \int_{0}^1 \nabla^2_{\mS} \wt{\Psi}^S(\mS_1 + st \Delta \mS; \mL)[t \Delta \mS] \,ds \nonumber \\
  & = \int_{0}^1 t\nabla^2_{\mS} \wt{\Psi}^S(\mS_1 + st \Delta \mS; \mL)[\Delta \mS] \,ds. \label{eq:FTC2}
\end{align}
Combining \eqref{eq:FTC1} and \eqref{eq:FTC2} yields
\begin{equation}\label{eq:bregman}
  \begin{aligned}
  \Breg{\wt{\Psi}^S(\cdot;\mL)}{\mS_2}{\mS_1} &= \wt{\Psi}^S(\mS_2; \mL) - \wt{\Psi}^S(\mS_1; \mL) - \langle \nabla_{\mS} \wt{\Psi}^S(\mS_1; \mL), \mS_2 - \mS_1 \rangle \\
  &= \int_0^1 \!\!\int_0^1
t\,
\nabla^2_{\mS} \wt{\Psi}^S(\mS_1 + st \Delta \mS; \mL)
[\Delta\mS,\Delta\mS]
\, ds\, dt .
\end{aligned}
\end{equation}
Given the above expression of Bregman divergence, to prove Property~\ref{item:smooth}, it suffices to show that 
\begin{equation}\label{eq:Hessian_bound}
  \nabla^2_{\mS} \wt{\Psi}^S(\mS; \mL)[\mD, \mD] \leq \beta \Tr(\mD^\top \mL^{-1} \mD), \quad \forall\, \mD \in \reals^{m \times n}.
\end{equation} 
We now prove \eqref{eq:Hessian_bound} in three steps; the proofs of several technical lemmas are deferred to Appendix~\ref{appen:wishart}.  Define the random matrices $\mG := \mS+\mL\mZ$ and $\mA := \mG\mG^\top$. Since $n>m$ and $\mZ \sim \mathcal{MN}(0, \mI_m, \mI_n)$, $\mG$ is full row rank almost surely.

\smallskip
\noindent\textbf{Step 1 (Hessian reduction to $\E[\mA^{-1/2}]$).}
Let $\Phi(\mS)=\|\mS\|_*$. If $\mS$ is full row rank, then $\Phi$ is twice differentiable at $\mS$ and
for all directions $\mD$,
\[
  \nabla^2\Phi(\mS)[\mD,\mD]
  \;\le\;
  \Tr\bigl(\mD^\top(\mS\mS^\top)^{-1/2}\mD\bigr)
  \qquad\text{(Lemma~\ref{lem:nuclear_appen}).}
\]
Using dominated convergence (or Leibniz' rule) to interchange expectation and differentiation, we obtain
\begin{equation}\label{eq:hessian_D_D_clean}
  \nabla^2_{\mS} \wt\Psi^S(\mS; \mL)[\mD,\mD]
  = \E\!\left[\nabla^2\Phi(\mS+\mL\mZ)[\mD,\mD]\right]
  \le \Tr\Bigl(\mD^\top \E[\mA^{-1/2}]\,\mD\Bigr).
\end{equation}

\smallskip
\noindent\textbf{Step 2 (bound $\E[\mA^{-1}]$ via a Wishart argument).}
Let $\mY:=\mL^{-1}\mS$ and $\tilde{\mA}:=(\mY+\mZ)(\mY+\mZ)^\top$.
Then 
$$\mA = (\mS + \mL \mZ)(\mS + \mL \mZ)^\top = \mL(\mL^{-1}\mS + \mZ)(\mL^{-1}\mS + \mZ)^\top \mL =  \mL\,\tilde{\mA}\,\mL,$$ 
and hence
\begin{equation}\label{eq:ainv_factor}
  \mA^{-1} = \mL^{-1}\tilde{\mA}^{-1}\mL^{-1}
  \quad\Rightarrow\quad
  \E[\mA^{-1}] = \mL^{-1}\E[\tilde{\mA}^{-1}]\mL^{-1}.
\end{equation}
Moreover, $\tilde{\mA}$ follows a (noncentral) Wishart distribution with $n$ degrees of freedom and identity scale.
In particular, it holds that
\begin{equation}\label{eq:tildeAinv_bound}
  \E[\tilde{\mA}^{-1}] \preceq \frac{1}{n-m-1}\mI
  \qquad\text{(Lemma~\ref{lem:inverse_wishart_appen}).}
\end{equation}
Combining \eqref{eq:ainv_factor} and \eqref{eq:tildeAinv_bound} yields
\begin{equation}\label{eq:Ainv_bound}
  \E[\mA^{-1}] \preceq \frac{1}{n-m-1}\,\mH^{-2}.
\end{equation}

\smallskip
\noindent\textbf{Step 3 (from $\E[\mA^{-1}]$ to $\E[\mA^{-1/2}]$).}
Since $t\mapsto \sqrt{t}$ is operator concave on $\semiS{m}$~\citep{Bhatia1997}, Jensen's inequality gives
\[
  \E[\mA^{-1/2}]
  = \E\!\bigl[(\mA^{-1})^{1/2}\bigr]
  \preceq \bigl(\E[\mA^{-1}]\bigr)^{1/2}.
\]
Using \eqref{eq:Ainv_bound} and the monotonicity of the matrix square root, we have
\[
  \bigl(\E[\mA^{-1}]\bigr)^{1/2}
  \preceq \left(\frac{1}{n-m-1}\mH^{-2}\right)^{1/2}
  = \frac{1}{\sqrt{n-m-1}}\mH^{-1}.
\]
Plugging this into \eqref{eq:hessian_D_D_clean} leads to
\[
  \nabla^2_{\mS} \wt\Psi^S(\mS; \mL)[\mD,\mD]
  \le \frac{1}{\sqrt{n-m-1}}\Tr(\mD^\top \mH^{-1}\mD).
\]
This completes the proof.

\subsubsection{Technical Matrix Lemmas}\label{appen:wishart}
In the following lemma, we characterize the derivative of the nuclear norm.
\begin{lemma}\label{lem:nuclear_appen}
  Consider the nuclear norm function $\Phi(\mX) = \|\mX\|_*$, where $\mX \in \mathbb{R}^{m\times n}$ with $m \leq n$. If $\mX$ is full row rank, then $\Phi$ is twice differentiable at $\mX$, and for all $\mD \in \mathbb{R}^{m \times n}$, 
  \begin{equation*}
    \nabla^2\Phi(\mX)[\mD, \mD] \leq \Tr(\mD^\top (\mX \mX^\top)^{-1/2} \mD). 
  \end{equation*}
\end{lemma}
\begin{proof}
  We begin by rewriting the nuclear norm as $\Phi(\mX) = \|\mX\|_* = \Tr(\sqrt{\mX \mX^\top})$. Define $\mA(\mX) = \mX \mX^\top$ and $g(\mA) = \Tr(\sqrt{\mA})$, so that $\Phi = g \circ \mA$. The mapping $\mA(\mX)$ is a polynomial function and hence twice differentiable everywhere. Moreover, since $\mX$ is full row rank, $\mA(\mX)$ is positive definite. The function $g$ is a spectral function induced by $t \mapsto t^{1/2}$, which is $C^2$ on $(0,\infty)$; therefore, $g$ is twice Fr\'echet differentiable over $\semiS{m}$~\citep{lewis2001twice}. Consequently, $\Phi$ is twice differentiable at $\mX$. 
  
  Moreover, applying the chain rule for Fr\'echet derivatives, we obtain: 
\begin{align}
    \nabla \Phi(\mX)[\mD] &=  \nabla g(\mA)[\nabla \mA(\mX)[\mD]] \\ 
    \nabla^2 \Phi(\mX)[\mD_1, \mD_2] &=  \nabla^2 g(\mA)[\nabla \mA(\mX)[\mD_1], \nabla \mA(\mX)[\mD_2] ] + \nabla g(\mA)[\nabla^2 \mA(\mX)[\mD_1, \mD_2]]. \label{eq:chain_rule}
  \end{align}
  We now compute the derivatives of $\mA$ explicitly. For any $\mD \in \mathbb{R}^{m \times n}$, we have $\nabla g(\mA)[\mD] \!= \frac{1}{2}\Tr(\mA^{-1/2} \mD)$, and for any $\mD_1$, $\mD_2 \in \mathbb{R}^{m \times n}$, we have $\nabla^2 \mA(\mX)[\mD_1, \mD_2] = \mD_1 \mD_2^\top + \mD_2 \mD_1^\top$. 
Next, we consider the derivatives of $g$. For any symmetric matrix $\mD$, $\nabla g(\mA)[\mD] = \frac{1}{2}\Tr(\mA^{-1/2} \mD)$. Furthermore, since the scalar function $t\mapsto t^{1/2}$ is operator concave~\citep{Bhatia1997}, the induced spectral function $g$ is concave on $\semiS{m}$. As a result, $\nabla^2 g(\mA)[\mD, \mD] \leq 0$ for all symmetric $\mD$. 
Specializing \eqref{eq:chain_rule} to $\mD_1 = \mD_2 = \mD$, we obtain 
  \begin{equation}\label{eq:quadratic_form}
    \nabla^2 \Phi(\mX)[\mD, \mD] =  \nabla^2 g(\mA)[\nabla \mA(\mX)[\mD], \nabla \mA(\mX)[\mD] ] + \nabla g(\mA)[\nabla^2 \mA(\mX)[\mD, \mD]].
  \end{equation}
  The first term in \eqref{eq:quadratic_form} is nonpositive by concavity of $g$. For the second term, we compute
  \begin{equation*}
    \nabla g(\mA)[\nabla^2 \mA(\mX)[\mD, \mD]] = \nabla g(\mA)[2\mD \mD^\top ] = \frac{1}{2}\Tr(\mA^{-1/2} 2\mD \mD^\top) = \Tr(\mD^\top \mA^{-1/2} \mD).
  \end{equation*}
  Combining these bounds and substituting $\mA = \mX\mX^\top$ yields the desired inequality. 
  \end{proof}

Before presenting our key lemma on a noncentral Wishart matrix in Lemma~\ref{lem:inverse_wishart_appen}, we first recall a basic property of the noncentral chi-square distribution~\citep{Muirhead2005}.
Recall that if $\vx\in\mathbb{R}^k$ satisfies $\vx\sim \mathcal{N}(\vmu,\mI_k)$, then
\(
\mX := \|\vx\|_2^2
\)
follows a noncentral chi-square distribution, denoted by $\chi_k^2(\lambda)$, with
degrees of freedom $k$ and noncentrality parameter $\lambda := \|\vmu\|_2^2$.

\begin{lemma}\label{lem:noncentral_chi}
Suppose $\mX\sim \chi_k^2(\lambda)$ with $k\ge 3$. Then
\[
  \E\!\left[\mX^{-1}\right] \;\le\; \frac{1}{{k-2}}.
\]
\end{lemma}

\begin{proof}
A noncentral chi-square admits a Poisson mixture representation~\cite[Corollary 1.3.5]{Muirhead2005}. Specifically, if $\mX\sim \chi_k^2(\lambda)$ and $\mJ\sim \mathrm{Pois}(\lambda/2)$,  then conditional on $\mJ=j$, we have
\(
\mX\,|\,(\mJ=j)\sim \chi^2_{k+2j}
\).
Hence, by the law of total expectation,
\begin{equation}\label{eq:poisson}
  \E\!\left[\mX^{-1}\right]
  = \E\!\left[\,\E\!\left[\mX^{-1}\mid \mJ\right]\right]
  = \E\!\left[\,\E\!\left[\left(\chi^2_{k+2\mJ}\right)^{-1}\right]\right].
\end{equation}
For $\mU\sim \chi^2_r$ with $r>1$, a direct calculation gives
\[
  \E[\mU^{-1}]
  = \frac{1}{2^{r/2}\Gamma(r/2)}\int_0^\infty x^{r/2-2}e^{-x/2}\,dx
  = \frac{\Gamma(\frac{r}{2}-1)}{2\,\Gamma(\frac{r}{2})} = \frac{1}{r-2}.
\]
Applying this with $r=k+2\mJ$ in~\eqref{eq:poisson} gives
\[
  \E\!\left[\mX^{-1}\right]
  = \E\!\left[\E\!\left[\left(\chi^2_{k+2\mJ}\right)^{-1}\right]\right]
  \le \E\!\left[\frac{1}{{k+2\mJ-2}}\right]
  \le \frac{1}{{k-2}},
\]
since $\mJ\ge 0$ almost surely and $t\mapsto 1/{t}$ is decreasing.
\end{proof}

 Using Lemma~\ref{lem:noncentral_chi}, we are ready to present our key result on the expected inverse of a Wishart matrix. Our proof is inspired by \cite{HillierKan2022}.  
  \begin{lemma}\label{lem:inverse_wishart_appen}
Let $\mZ\in\mathbb{R}^{m\times n}$ have i.i.d.\ $\mathcal{N}(0,1)$ entries and let $\mY\in\mathbb{R}^{m\times n}$ be deterministic.
Assume $n\ge m+2$ and define
\(
\mA_{\mY} := (\mZ+\mY)(\mZ+\mY)^\top \in \mathbb{R}^{m\times m}.
\)
Then, for any $\mY$,
\begin{equation}\label{eq:inverse_wishart_bound}
  \E[\mA_{\mY}^{-1}] \;\preceq\; \frac{1}{{\,n-m-1\,}}\mI_m.
\end{equation}
\end{lemma}

\begin{proof}
\textbf{Step 1: reduce to diagonal $\mY$.}
Let $\mY=\mU\mSigma\mV^\top$ be a singular value decomposition, where $\mU\in\mathbb{R}^{m\times m}$ and $\mV\in\mathbb{R}^{n\times n}$ are orthogonal, and
\(
\mSigma=[\mathrm{diag}(\sigma_1,\dots,\sigma_m)\ \ 0]\in\mathbb{R}^{m\times n}.
\)
Let $\widetilde{\mZ}:=\mU^\top \mZ \mV$. By orthogonal invariance of the standard Gaussian, $\widetilde{\mZ}\stackrel{d}{=}\mZ$, and
\[
  \mA_{\mY}
  = (\mZ+\mU\mSigma\mV^\top)(\mZ+\mU\mSigma\mV^\top)^\top
  = \mU(\widetilde{\mZ}+\mSigma)(\widetilde{\mZ}+\mSigma)^\top \mU^\top.
\]
Hence $\mA_{\mY}^{-1}=\mU\,\mA_{\mSigma}^{-1}\mU^\top$ with $\mA_{\mSigma}:=(\widetilde{\mZ}+\mSigma)(\widetilde{\mZ}+\mSigma)^\top$, and taking expectations gives
\[
  \E[\mA_{\mY}^{-1}] = \mU\,\E[\mA_{\mSigma}^{-1}]\,\mU^\top.
\]
Since $\mU$ is orthogonal, to prove~\eqref{eq:inverse_wishart_bound} it suffices to bound $\E[\mA_{\mSigma}^{-1}]$.

\noindent\textbf{Step 2: $\E[\mA_{\mSigma}^{-1}]$ is diagonal.}
Let $\mathcal{T}:=\{\mT=\mathrm{diag}(\pm1,\dots,\pm1)\in\mathbb{R}^{m\times m}\}$ denote the set of all sign-change matrices, and define
\(
\widetilde{\mT}:=\mathrm{diag}(\mT,\mI_{n-m})\in\mathbb{R}^{n\times n}.
\)
Since $\mSigma$ has a diagonal left block, we have $\mT\mSigma\widetilde{\mT}=\mSigma$.
Moreover, $\mT\widetilde{\mZ}\widetilde{\mT}\stackrel{d}{=}\widetilde{\mZ}$, and therefore
\[
  \mA_{\mSigma}
  = (\widetilde{\mZ}+\mSigma)(\widetilde{\mZ}+\mSigma)^\top =(\tilde{\mZ} +
  \mT\mSigma \tilde{\mT}) (\tilde{\mZ} + 
  \mT \mSigma \tilde{\mT})^\top 
  \stackrel{d}{=}
  \mT\,\mA_{\mSigma}\,\mT.
\]
Inverting the matrices and taking expectations yields
\(
\E[\mA_{\mSigma}^{-1}] = \mT\,\E[\mA_{\mSigma}^{-1}]\,\mT
\)
for all $\mT\in\mathcal{T}$, which forces $\E[\mA_{\mSigma}^{-1}]$ to be diagonal.

\noindent\textbf{Step 3: bound each diagonal entry.}
Since $\E[\mA_{\mSigma}^{-1}]$ is diagonal, it is enough to show
\(
\E[(\mA_{\mSigma}^{-1})_{ii}] \le \frac{1}{\sqrt{n-m-1}}
\)
for each $i$. Fix $i=1$ (the others are identical by relabeling rows).
Write the first row of $\widetilde{\mZ}+\mSigma$ as $\tilde{\vz}_1^\top\in\mathbb{R}^n$ and the remaining rows as
$\widetilde{\mZ}_{-1}\in\mathbb{R}^{(m-1)\times n}$, so that
\(
\widetilde{\mZ}+\mSigma = \begin{bmatrix}\tilde{\vz}_1^\top\\ \widetilde{\mZ}_{-1}\end{bmatrix}.
\)
A Schur complement computation gives
\[
  (\mA_{\mSigma}^{-1})_{11}
  = \frac{1}{
    \tilde{\vz}_1^\top\tilde{\vz}_1
    - \tilde{\vz}_1^\top \widetilde{\mZ}_{-1}^\top(\widetilde{\mZ}_{-1}\widetilde{\mZ}_{-1}^\top)^{-1}\widetilde{\mZ}_{-1}\tilde{\vz}_1
  }
  = \frac{1}{\tilde{\vz}_1^\top \mP\,\tilde{\vz}_1},
\]
where
\(
\mP := \mI_n - \widetilde{\mZ}_{-1}^\top(\widetilde{\mZ}_{-1}\widetilde{\mZ}_{-1}^\top)^{-1}\widetilde{\mZ}_{-1}
\)
is the orthogonal projector onto $\mathrm{Row}(\widetilde{\mZ}_{-1})^\perp$.
Almost surely $\mathrm{rank}(\widetilde{\mZ}_{-1})=m-1$, hence $\mP$ has rank $n-m+1$.

Conditioned on $\widetilde{\mZ}_{-1}$, the vector $\tilde{\vz}_1\sim \mathcal{N}(\sigma_1\ve_1,\mI_n)$ is independent of $\widetilde{\mZ}_{-1}$, and thus
\(
\tilde{\vz}_1^\top \mP \tilde{\vz}_1
\)
is a noncentral chi-square random variable with $k=n-m+1$ degrees of freedom (and some noncentrality parameter depending on $\mP$ and $\sigma_1$). Since $n\ge m+2$, we have $k\ge 3$, and Lemma~\ref{lem:noncentral_chi} implies
\[
  \E\!\left[{(\mA_{\mSigma}^{-1})_{11}} \,\middle|\, \widetilde{\mZ}_{-1}\right]
  = \E\!\left[(\tilde{\vz}_1^\top \mP \tilde{\vz}_1)\,\middle|\, \widetilde{\mZ}_{-1}\right]
  \le \frac{1}{{k-2}}
  = \frac{1}{{n-m-1}}.
\]
Taking expectations over $\widetilde{\mZ}_{-1}$ yields
$
  \E[(\mA_{\mSigma}^{-1})_{11}]\le \frac{1}{{n-m-1}}.
$
Therefore $\E[\mA_{\mSigma}^{-1}] \preceq \frac{1}{{n-m-1}}\mI_m$, and conjugating by $\mU$ completes the proof.
\end{proof}

\subsection{Proofs for FAML (Theorem~\ref{thm:FAML})}\label{appen:faml}
In the section, we verify that the hyperbolic smoothing potential family $\wt{\Psi}^{H}(\mS; \mL)$ satisfies all the properties in Definition~\ref{def:admissible_potential} with $(\alpha, \beta) = (1,1)$.

\medskip
\mypara{Feasibility} Recall from Section~\ref{subsec:FAML} that $\nabla_{\mS} \wt{\Psi}^{H}(\mS;\mL) = \left({\mS \mS^\top + \mL \mL^\top}\right)^{-1/2} \mS$. Since $\mL\mL^\top \succeq 0$, it holds that 
\begin{equation*}
  \nabla_{\mS} \wt{\Psi}^{H}(\mS;\mL) \nabla_{\mS} \wt{\Psi}^{H}(\mS;\mL)^\top = \left({\mS \mS^\top + \mL \mL^\top}\right)^{-1/2} \mS \mS^\top \left({\mS \mS^\top + \mL \mL^\top}\right)^{-1/2} \preceq \mI.
\end{equation*}
This implies $\|\nabla_{\mS} \wt{\Psi}^{H}(\mS;\mL)\|_{\op} \leq 1$, verifying Property~\ref{item:feas}. 

\medskip
\mypara{Dominance} It is easy to see that $\wt{\Psi}^{H}(\mS;\mzero)=\tr\!\bigl(\sqrt{\mS \mS^\top}\bigr) = \|\mS\|_*$. Moreover, since $\mL \mL^\top \succeq 0$, we have $\sqrt{\mS \mS^\top+\mL \mL^\top} \succeq \sqrt{\mS \mS^\top}$, yielding $\wt{\Psi}^{H}(\mS;\mL) \geq \|\mS\|_*$.

\medskip
\mypara{Upper stability} We rewrite the potential as $\wt{\Psi}^{H}(\mS;\mL) = \|\begin{bmatrix}
        \mS & \mL
\end{bmatrix}\|_*$. Then by triangle inequality, for any $\mL_1 \preceq \mL_2$ and $\mX \in \reals^{m \times n}$, we have 
\begin{equation*}
  \wt{\Psi}^{H}(\mS;\mL_2) - \wt{\Psi}^{H}(\mS;\mL_1) = \left\|\begin{bmatrix}
        \mS & \mL_2
\end{bmatrix}\right\|_* - \left\|\begin{bmatrix}
        \mS & \mL_1
\end{bmatrix}\right\|_* \leq \left\|\begin{bmatrix}
        \mzero & \mL_2 - \mL_1
\end{bmatrix}\right\|_* = \|\mL_2 - \mL_1\|_*.
\end{equation*}
Since $\mL_2 -\mL_1 \succeq 0$, we further have $\|\mL_2 - \mL_1\|_* = \tr(\mL_2) - \tr(\mL_1)$. This proves Property~\ref{item:stable} with $\alpha = 1$. 

\medskip
\mypara{Smoothness} We rely on the characterization of the nuclear norm in Lemma~\ref{lem:nuclear_appen}. For any $\mL \succ 0$ and $\mX \succeq 0$, the augmented matrix $\begin{bmatrix}
        \mS & \mL
\end{bmatrix}$ is full row rank. Hence, $\wt{\Psi}^{H}(\mS;\mL) = \|\begin{bmatrix}
        \mS & \mL
\end{bmatrix}\|_*$ is twice differentiable and it follows from Lemma~\ref{lem:nuclear_appen} that for all $\mD \in \reals^{m\times n}$, 
\begin{equation*}
  \nabla^2_{\mS}\wt{\Psi}^{H}(\mS;\mL)[\mD ,\mD] \leq \tr(\mD^\top(\mS\mS^\top + \mL\mL^\top)^{-1/2}\mD).
\end{equation*}
Since $\mS \mS^\top \succeq 0$ and note that $\mL$ is positive definite, this further implies that $$\nabla^2_{\mS}\wt{\Psi}^{H}(\mS;\mL)[\mD ,\mD] \leq \tr(\mD^\top(\mL\mL^\top)^{-1/2}\mD) = \Tr(\mD \mL^{-1}\mD).$$ Hence, we follow the same argument as in Theorem~\ref{thm:regret_FTPL} and use \eqref{eq:bregman} to conclude that Property~\ref{item:smooth} is satisfied with $\beta = 1$.

\section{The Cost of Solving Subproblems}\label{appen:subproblem_complexity}
\subsection{Shampoo and One-Sided Shampoo}\label{appen:projection}

For Shampoo and one-sided Shampoo, each iteration requires solving the quadratic projection subproblem in \eqref{eq:shampoo} over the operator-norm ball. As mentioned in Section~\ref{sec:adaptive_matrix_OL}, this typically necessitates an iterative inner solver, since the operator-norm geometry precludes a simple closed-form update. Concretely, Euclidean projection onto the operator-norm ball is equivalent to projecting the singular values, which in general requires computing a full SVD; likewise, first-order inner methods such as Frank-Wolfe rely on a linear minimization oracle over the operator-norm ball, which amounts to computing a polar factor.

A further complication is that the quadratic subproblem is shaped by the adaptive preconditioners $\mL_t$ and $\mR_t$, which can be ill-conditioned. As a result, inner-loop methods can converge slowly at a sublinear rate. In practice this translates into many inner iterations, and hence multiple SVD or polar computations \emph{per outer step}, which are inherently sequential and can dominate the wall-clock cost.

In contrast, as detailed in the next two sections, our FTPL-based method performs its spectral primitive (polar computation) through a single randomized smoothing step that is naturally parallelizable, while our FAML-based method requires only one matrix inverse square root (or equivalently a single polar computation on an augmented matrix) per iteration.

\subsection{FTPL}\label{appen:cost_FTPL}
As discussed in Section~\ref{subsec:FTPL}, each iteration of the update in~\eqref{eq:matrix_ftpl_mc} computes
$\mG_t \mG_t^\top$, followed by a Cholesky factorization and \(k\) parallel computations of
matrix polar factors. Below we describe a concrete implementation of the polar factor via the Newton--Schulz
iteration. Given an input matrix \(\mS \in \mathbb{R}^{m \times n}\), initialize and iterate
\begin{equation}\label{eq:NS}
  \mX^{(0)} = \frac{\mS}{\|\mS\|_F}, 
  \qquad 
  \mX^{(i+1)} = \tfrac{1}{2}\bigl(3\mI - \mX^{(i)}(\mX^{(i)})^\top\bigr)\mX^{(i)} .
\end{equation}
It is known that the iterates converge (locally) quadratically to \(\mathrm{polar}(\mS)\), and in particular the iteration converges provided the singular values of \(\mX^{(0)}\) lie in \((0,\sqrt{3})\)~\citep{higham2008functions}.
Each Newton--Schulz step is dominated by two matrix--matrix multiplications, for a leading cost
\(2 \times (2m^2n) = 4m^2n\) floating-point operations per step.

Let \(s_{\text{par}}\ge 1\) denote the effective parallel speedup for the \(k\) polar-factor computations. Putting these components together, the leading-order \emph{wall-clock} cost per-iteration is
\begin{equation*}
  \underbrace{2m^2n}_{\mG_t\mG_t^\top}
  \;+\;
  \underbrace{\tfrac{1}{3}m^3}_{\text{Cholesky}}
  \;+\;
  \underbrace{\frac{4kK\,m^2n}{s_{\text{par}}}}_{\substack{K\text{ Newton--Schulz steps}\\ \text{for each of }k\text{ polar factors}}}\,,
\end{equation*}
where \(K\) denotes the number of Newton--Schulz iterations used to compute each polar factor.
\subsection{FAML}\label{appen:cost_FAML}

The first implementation based on~\eqref{eq:FAML_update_equiv} requires computing a matrix inverse square root.
We describe a concrete realization using the coupled Newton--Schulz iteration.
Given an input matrix \(\mA \in \semiS{m}\), initialize
\[
  \mY^{(0)} = \frac{\mA}{\sqrt{\|\mA\|_F}},
  \qquad
  \mZ^{(0)} = \frac{\mI}{\sqrt{\|\mA\|_F}},
\]
and iterate
\begin{equation*}
   \mT^{(i)} = \mZ^{(i)} \mY^{(i)}, \qquad
   \mY^{(i+1)} = \tfrac{1}{2}\mY^{(i)}(3\mI - \mT^{(i)}), \qquad
   \mZ^{(i+1)} = \tfrac{1}{2}(3\mI-\mT^{(i)})\mZ^{(i)} .
\end{equation*}
The coupled iteration converges provided
\(\|\mI - \mA / \|\mA\|_F\|_{\op} \le 1\).
Each Newton--Schulz step is dominated by three \(m\times m\) matrix--matrix
multiplications, resulting in \(6m^3\) floating-point operations per iteration.

In each outer iteration of the algorithm, the following operations are performed:
\begin{itemize}
    \item Compute \(\mG_t \mG_t^\top\): \(2m^2n\) flops;
    \item Compute \(\mS_t \mS_t^\top\): \(2m^2n\) flops;
    \item Coupled Newton--Schulz iteration: \(6m^3\) flops per step;
    \item Final preconditioning: \(2m^2n\) flops.
\end{itemize}
Consequently, the leading-order per-iteration cost is
\[
  6m^2n + 6K m^3,
\]
where \(K\) denotes the number of coupled Newton--Schulz iterations.

Alternatively, the update in~\eqref{eq:FAML_update} can be implemented by computing the polar factor of the
augmented matrix
\(
\widehat{\mS}_t \;=\; \begin{bmatrix} \mS_t & \tfrac{1}{\eta}\mL_t \end{bmatrix}.
\)
We describe a Newton--Schulz-based implementation that exploits the block structure of \(\widehat{\mS}_t\).
Since the final update only uses the leading block of the polar factor, the procedure can be specialized
to avoid explicitly forming \(\mL_t\). In particular, for the choice of \(\mL_t\) in~\eqref{eq:L_t},
this avoids computing a matrix square root such as \((G^2\mI+\mM_t)^{1/2}\).

Let
\(
\widehat{\mS} \;=\; \begin{bmatrix} \mS & \mL \end{bmatrix}
\)
denote the input, and maintain iterates
\(
\widehat{\mX}^{(i)} \;=\; \begin{bmatrix} \mX^{(i)} & \mY^{(i)} \end{bmatrix}
\)
at step \(i\) of the Newton--Schulz iteration. Applying the standard update~\eqref{eq:NS} gives
\[
\widehat{\mX}^{(i+1)}
\;=\;
\tfrac{1}{2}\bigl(3\mI - \widehat{\mX}^{(i)}(\widehat{\mX}^{(i)})^\top\bigr)\widehat{\mX}^{(i)} .
\]
Writing this update in block form yields
\[
\mT^{(i)} \;=\; \tfrac{1}{2}\Bigl(3\mI
- \mX^{(i)}(\mX^{(i)})^\top
- \mY^{(i)}(\mY^{(i)})^\top\Bigr), \qquad
\mX^{(i+1)} = \mT^{(i)}\mX^{(i)}, \qquad
\mY^{(i+1)} = \mT^{(i)}\mY^{(i)} .
\]
Crucially, the update of \(\mX^{(i)}\) depends on \(\mY^{(i)}\) only through the Gram matrix
\(\mX^{(i)}(\mX^{(i)})^\top + \mY^{(i)}(\mY^{(i)})^\top\).
We therefore eliminate \(\mY^{(i)}\) by defining
\[
\mB^{(i)} \;\coloneqq\; \mX^{(i)}(\mX^{(i)})^\top + \mY^{(i)}(\mY^{(i)})^\top .
\]
Since
\[
\mB^{(i+1)}
= \mX^{(i+1)}(\mX^{(i+1)})^\top + \mY^{(i+1)}(\mY^{(i+1)})^\top
= \mT^{(i)} \mB^{(i)} \mT^{(i)},
\]
we obtain the equivalent recursion
\begin{equation}\label{eq:NS_leon}
  \mT^{(i)} = \tfrac{1}{2}\bigl(3\mI - \mB^{(i)}\bigr), \qquad
  \mX^{(i+1)} = \mT^{(i)} \mX^{(i)}, \qquad
  \mB^{(i+1)} = \mT^{(i)} \mB^{(i)} \mT^{(i)} .
\end{equation}

For initialization, note that
\[
\|\widehat{\mS}\|_F^2 \;=\; \|\mS\|_F^2 + \|\mL\|_F^2 \;=\; \|\mS\|_F^2 + \tr(\mL\mL^\top),
\]
and set
\[
\mX^{(0)} = \frac{\mS}{\|\widehat{\mS}\|_F}, \qquad
\mB^{(0)} = \frac{\widehat{\mS}\widehat{\mS}^\top}{\|\widehat{\mS}\|_F^2}
= \frac{\mS\mS^\top + \mL\mL^\top}{\|\widehat{\mS}\|_F^2}.
\]
Moreover, with the choice of \(\mL_t\) in~\eqref{eq:L_t}, we have \(\mL_t\mL_t^\top = G^2\mI + \mM_t\),
so \(\mL_t\mL_t^\top\) can be formed directly from \(\mM_t\coloneqq \sum_{s=1}^t \mG_s\mG_s^\top\)
without explicitly computing \((G^2\mI+\mM_t)^{1/2}\).

Each Newton--Schulz step in~\eqref{eq:NS_leon} requires one \(m\times m\) by \(m\times n\) multiplication
(to update \(\mX^{(i+1)}\)) and two \(m\times m\) by \(m\times m\) multiplications
(to update \(\mB^{(i+1)}\)), for a total of \(2m^2n + 4m^3\) floating-point operations.
Including the initialization cost of forming \(\mS\mS^\top\) and \(\mG_t\mG_t^\top\) (each \(2m^2n\)),
the resulting leading-order per-iteration cost is
\[
4m^2n \;+\; K\,(2m^2n + 4m^3),
\]
where \(K\) denotes the number of Newton--Schulz iterations.

\section{Proofs for Online-to-nonconvex Conversion}\label{appen:o2nc}
We first describe the O2NC reduction protocol in full. At each iteration $t$, the environment draws a sample $\zeta_t\sim\mathcal{D}$ and reveals the corresponding stochastic gradient $\mG_t$. This gradient induces a (discounted) linear loss, which we feed to an online learner $\mathcal{A}$ to obtain the next update direction. Finally, we update the parameters using an exponentially distributed step size:
\begin{align}
\mG_t &= \nabla_{\mW}\ell(\mW_t;\zeta_t), \qquad \zeta_t\sim\mathcal{D}, \label{eq:o2nc_grad}\\
\ell_t^{[\beta]}(\mX) &\coloneq \beta^{-t}\,\langle \mG_t,\mX\rangle,
\qquad \text{(loss revealed)} \label{eq:o2nc_loss}\\
\mX_{t+1} &= \mathcal{A}\!\left(\ell_t^{[\beta]}\right)\in\mathcal{X},
\qquad \text{(learner selection)} \label{eq:o2nc_dir}\\
\mW_{t+1} &= \mW_t + s_{t+1}\mX_{t+1}, \qquad s_{t+1}\sim \mathrm{Exp}(1).
\label{eq:o2nc_update}
\end{align}
The exponential step-size choice is motivated by the \emph{Random Scaling Lemma}~\citep[Lemma~3.1]{zhang2024random}, which connects the expected objective decrease to the linearized update:
\begin{equation}\label{eq:random_scaling}
\E_{s_{t+1}}\!\bigl[L(\mW_{t+1})-L(\mW_t)\bigr]
= \E_{s_{t+1},\zeta_{t+1}}\!\bigl[\langle \mG_{t+1},\mX_{t+1}\rangle\bigr].
\end{equation}
For ease of exposition, we also provide pseudocode for Muon, Pion, and Leon in Algorithms~\ref{alg:muon}, \ref{alg:pion}, and \ref{alg:leon}, respectively.
\begin{algorithm}[!t]
        \caption{Muon}\label{alg:muon}
            \centering 
            \begin{algorithmic}[1]\small
\FOR{iteration $t=1,\dots,T$}
\STATE $\mG_t \leftarrow \nabla_{\mW}\ell(\mW_t; \zeta_t) \in \mathbb{R}^{m \times n}$
          \STATE $\hat{\mG}_t \leftarrow \beta \hat{\mG}_{t-1} + \mG_t$
          \STATE $\mP_t \leftarrow \mathrm{polar}(\hat{\mG}_t)$ 
          \STATE $\mW_{t+1} \leftarrow \mW_t - \alpha_t \mP_t$
          \ENDFOR
            \end{algorithmic}
        \end{algorithm}
  
        \begin{algorithm}[!t]
        \caption{Pion}\label{alg:pion}
            \centering 
            \begin{algorithmic}[1]\small
\FOR{iteration $t=1,\dots,T$}
\STATE $\mG_t \leftarrow \nabla_{\mW}\ell(\mW_t; \zeta_t) \in \mathbb{R}^{m \times n}$
          \STATE $\hat{\mG}_t \leftarrow \beta_1 \hat{\mG}_{t-1} + \mG_t$
          \STATE {$\mM_t \leftarrow \beta_2 \mM_{t-1} + \mG_t \mG_t^\top $}
          \STATE Cholesky factorization $\mM_t = \mL_t \mL_t^\top $
          \STATE {$\mP_t \leftarrow \frac{1}{k}\sum_{i=1}^k \mathrm{polar}(\hat{\mG}_t + \mL_t \mZ^{(i)}_t)$, where $\mZ_t^{(i)} \sim \mathcal{MN}(0, \mI_m, \mI_n)$}
          \STATE $\mW_{t+1} \leftarrow \mW_t - \alpha_t \mP_t$
          \ENDFOR
            \end{algorithmic}
        \end{algorithm}
        \begin{algorithm}[!t]
        \caption{Leon}\label{alg:leon}
            \centering 
            \begin{algorithmic}[1]\small
\FOR{iteration $t=1,\dots,T$}
\STATE $\mG_t \leftarrow \nabla_{\mW}\ell(\mW_t; \zeta_t) \in \mathbb{R}^{m \times n}$
          \STATE $\hat{\mG}_t \leftarrow \beta_1 \hat{\mG}_{t-1} + \mG_t$
          \STATE {$\mM_t \leftarrow \beta_2 \mM_{t-1} + \mG_t \mG_t^\top $}
          \STATE {$\mP_t \leftarrow \left( \hat{\mG}_t \hat{\mG}_t^\top + \mM_t\right)^{-1/2}\hat{\mG}_t$}
          \STATE $\mW_{t+1} \leftarrow \mW_t - \alpha_t \mP_t$
          \ENDFOR
            \end{algorithmic}
        \end{algorithm}

\subsection{Proof of Proposition~\ref{prop:o2nc}}\label{appen:o2nc_formal}

\begin{proposition}[Formal O2NC bound]\label{prop:o2nc_formal}
Define the exponentially weighted average (EWA)
\[
\bar{\mW}_t \coloneq \frac{1-\beta}{1-\beta^{t}}\sum_{s=1}^{t}\beta^{t-s}\mW_s,
\qquad t=1,\dots,T,
\]
and the random time index $\tau \in \{1,\dots,T\}$ with
\[
\Pr(\tau=t)=
\begin{cases}
\frac{1-\beta^{t}}{T}, & t=1,\dots,T-1,\\[2mm]
\frac{1-\beta^{T}}{(1-\beta)T}, & t=T.
\end{cases}
\]
Let $\mE_t \coloneq \mG_t-\nabla L(\mW_t)$ denote the stochastic noise, and let
\[
\Reg_t^{[\beta]}(D)\coloneq \max_{\|\mX\|\le D}\sum_{s=1}^{t}\beta^{t-s}\,\langle \mG_s,\mX_s-\mX\rangle
\]
be the discounted regret (with radius $D$). If
\(
D=\frac{1-\beta}{2\beta}\rho,
\)
then the expected $\rho$-stationarity gap at $\bar{\mW}_\tau$ satisfies
\begin{align}
\E_{\tau}\!\bigl[\|\nabla L(\bar{\mW}_{\tau})\|_\dagger^{[\rho]}\bigr]
&\le \frac{L(\mW_0)-L(\mW^*)}{(1-\beta)\rho\,T}
+ \frac{1}{T}\E\!\Bigl[\Reg_T^{[\beta]}(1) + (1-\beta)\sum_{t=1}^{T-1}\Reg_t^{[\beta]}(1)\Bigr] \nonumber\\
&\quad + \frac{1}{T}\E\!\Bigl[\Bigl\|\sum_{t=1}^{T}\beta^{T-t}\mE_t\Bigr\|\dual\Bigr]
+ \frac{1-\beta}{T}\sum_{t=1}^{T-1}\E\!\Bigl[\Bigl\|\sum_{s=1}^{t}\beta^{t-s}\mE_s\Bigr\|\dual\Bigr].
\label{eq:o2nc_formal_bound}
\end{align}
\end{proposition}

Our proof for Proposition~\ref{prop:o2nc_formal} is built on the following decomposition. 
\begin{lemma}[{\citealp[Appendix A.1]{ahn2025general}}]\label{lem:decomp}
For any $\beta\in(0,1)$,
\[
L(\mW_T)-L(\mW_0)
=
\sum_{t=1}^{T}\beta^{T-t}\bigl(L(\mW_t)-L(\mW_{t-1})\bigr)
+(1-\beta)\sum_{t=1}^{T-1}\sum_{s=1}^{t}\beta^{t-s}\bigl(L(\mW_s)-L(\mW_{s-1})\bigr).
\]
\end{lemma}

\mypara{Proof of Proposition~\ref{prop:o2nc_formal}}
We first relate function decrease to discounted gradients, regret, and noise.
Fix any $t\in[T]$ and any comparator $\mX$ with $\|\mX\|\le D$.
Using the update $\mW_s=\mW_{s-1}+s_s\mX_s$ and the identity in \eqref{eq:random_scaling},
\[
\E\bigl[L(\mW_s)-L(\mW_{s-1})\bigr]=\E\langle \mG_s,\mX_s\rangle
=\E\langle \nabla L(\mW_s),\mX\rangle + \E\langle \mG_s,\mX_s-\mX\rangle + \E\langle \mE_s,\mX\rangle.
\]
Multiplying by $\beta^{t-s}$ and summing over $s=1,\dots,t$ gives
\begin{align*}
\E\Bigl[\sum_{s=1}^{t}\beta^{t-s}\bigl(L(\mW_s)-L(\mW_{s-1})\bigr)\Bigr]
&= \E\Bigl\langle \sum_{s=1}^{t}\beta^{t-s}\nabla L(\mW_s),\,\mX\Bigr\rangle
+ \E\Bigl[\sum_{s=1}^{t}\beta^{t-s}\langle \mG_s,\mX_s-\mX\rangle\Bigr] \\
&\quad + \E\Bigl\langle \sum_{s=1}^{t}\beta^{t-s}\mE_s,\,\mX\Bigr\rangle.
\end{align*}
By duality, for any matrix $\mA$ and any $\|\mX\|\le D$,
$\langle \mA,\mX\rangle \ge - D\|\mA\|_\dagger$ and
$\langle \mB,\mX\rangle \le D\|\mB\|\dual$.
Therefore,
\begin{align}
\E\Bigl[\sum_{s=1}^{t}\beta^{t-s}\bigl(L(\mW_s)-L(\mW_{s-1})\bigr)\Bigr]
&\le -D\,\E\Bigl\|\sum_{s=1}^{t}\beta^{t-s}\nabla L(\mW_s)\Bigr\|_\dagger
+ \E\bigl[\Reg_t^{[\beta]}(D)\bigr] \nonumber\\
&\quad + D\,\E\Bigl\|\sum_{s=1}^{t}\beta^{t-s}\mE_s\Bigr\|\dual.
\label{eq:one_t_decrease}
\end{align}
Apply \eqref{eq:one_t_decrease} with $t=T$ and also with each $t=1,\dots,T-1$,
then combine them using Lemma~\ref{lem:decomp}. Rearranging yields
\begin{align}
&D\,\E\Bigl\|\sum_{t=1}^{T}\beta^{T-t}\nabla L(\mW_t)\Bigr\|_\dagger
+(1-\beta)D\sum_{t=1}^{T-1}\E\Bigl\|\sum_{s=1}^{t}\beta^{t-s}\nabla L(\mW_s)\Bigr\|_\dagger
\nonumber\\
&\le \E\bigl[L(\mW_0)-L(\mW_T)\bigr]
+\E\bigl[\Reg_T^{[\beta]}(D)\bigr]
+(1-\beta)\sum_{t=1}^{T-1}\E\bigl[\Reg_t^{[\beta]}(D)\bigr]
\nonumber\\
&\quad
+ D\,\E\Bigl\|\sum_{t=1}^{T}\beta^{T-t}\mE_t\Bigr\|\dual
+(1-\beta)D\sum_{t=1}^{T-1}\E\Bigl\|\sum_{s=1}^{t}\beta^{t-s}\mE_s\Bigr\|\dual.
\label{eq:after_decomp}
\end{align}

Next we convert the discounted-gradient terms into a stationarity guarantee for an EWA iterate.
For each $t\in[T]$, define a random iterate $\mY_t$ supported on $\{\mW_1,\dots,\mW_t\}$ by
\[
\Pr(\mY_t=\mW_s)=\frac{1-\beta}{1-\beta^{t}}\beta^{t-s},
\qquad s\in[t],
\]
so that $\E[\mY_t]=\bar{\mW}_t$.
We will use the following concentration-of-the-mean bound.

\begin{lemma}\label{lem:Y_barW_close}
For every $t\in[T]$, we have $\E\|\mY_t-\bar{\mW}_t\|\le \frac{2\beta}{1-\beta}D$.
\end{lemma}

\begin{proof}
Fix $t$ and define $q_{t,s}\coloneq \frac{1-\beta}{1-\beta^{t}}\beta^{t-s}$.
Let $\widehat{\mY}_t$ be an independent copy of $\mY_t$. By Jensen,
\[
\E\|\mY_t-\bar{\mW}_t\|
=\E\bigl\|\mY_t-\E[\widehat{\mY}_t]\bigr\|
\le \E\|\mY_t-\widehat{\mY}_t\|
=2\sum_{i=1}^{t}\sum_{j=1}^{i-1}q_{t,i}q_{t,j}\|\mW_i-\mW_j\|.
\]
By the triangle inequality,
$\|\mW_i-\mW_j\|\le \sum_{s=j+1}^{i}\|\mW_s-\mW_{s-1}\|$.
Thus
\[
\E\|\mY_t-\bar{\mW}_t\|
\le 2\sum_{s=2}^{t}\Bigl(\sum_{i=s}^{t}\sum_{j=1}^{s-1}q_{t,i}q_{t,j}\Bigr)\,\|\mW_s-\mW_{s-1}\|.
\]
Moreover,
\[
\sum_{i=s}^{t}\sum_{j=1}^{s-1}q_{t,i}q_{t,j}
=\Bigl(\frac{1-\beta}{1-\beta^{t}}\Bigr)^2
\Bigl(\sum_{i=s}^{t}\beta^{t-i}\Bigr)\Bigl(\sum_{j=1}^{s-1}\beta^{t-j}\Bigr)
\le \frac{\beta^{t-s+1}}{1-\beta^{t}}.
\]
Taking expectations and using $\E\|\mW_s-\mW_{s-1}\|=\E\|s_s\mX_s\|\le D$,
\[
\E\|\mY_t-\bar{\mW}_t\|
\le 2\sum_{s=2}^{t}\frac{\beta^{t-s+1}}{1-\beta^{t}}D
\le \frac{2\beta}{1-\beta}D.
\]
\end{proof}

Now take $D=\frac{1-\beta}{2\beta}\rho$ so Lemma~\ref{lem:Y_barW_close} gives
$\E\|\mY_t-\bar{\mW}_t\|\le \rho$.
By the definition of the $\rho$-stationarity gap,
this implies
\[
\|\nabla L(\bar{\mW}_t)\|_\dagger^{[\rho]}
\le \bigl\|\E[\nabla L(\mY_t)]\bigr\|_\dagger.
\]
By Jensen and the definition of $\mY_t$,
\[
\bigl\|\E[\nabla L(\mY_t)]\bigr\|_\dagger
\le \E\Bigl\|\nabla L(\mY_t)\Bigr\|_\dagger
\le \frac{1-\beta}{1-\beta^{t}}\,\E\Bigl\|\sum_{s=1}^{t}\beta^{t-s}\nabla L(\mW_s)\Bigr\|_\dagger.
\]
Plugging this into \eqref{eq:after_decomp} yields
\begin{align*}
&\frac{1-\beta^{T}}{1-\beta}\E\bigl[\|\nabla L(\bar{\mW}_{T})\|_\dagger^{[\rho]}\bigr]
+ \sum_{t=1}^{T-1}(1-\beta^{t})\,\E\bigl[\|\nabla L(\bar{\mW}_{t})\|_\dagger^{[\rho]}\bigr] \\
&\le \frac{1}{D}\E\bigl[L(\mW_0)-L(\mW_T)\bigr]
+\frac{1}{D}\E\bigl[\Reg_T^{[\beta]}(D)\bigr]
+\frac{1-\beta}{D}\sum_{t=1}^{T-1}\E\bigl[\Reg_t^{[\beta]}(D)\bigr] \\
&\quad
+\E\Bigl\|\sum_{t=1}^{T}\beta^{T-t}\mE_t\Bigr\|\dual
+(1-\beta)\sum_{t=1}^{T-1}\E\Bigl\|\sum_{s=1}^{t}\beta^{t-s}\mE_s\Bigr\|\dual.
\end{align*}
Note that by a scaling argument, we have $\frac{1}{D}\Reg_t^{[\beta]}(D) = \Reg_t^{[\beta]}(1)$.
Finally, by the definition of $\tau$,
\[
\E_{\tau}\bigl[\|\nabla L(\bar{\mW}_{\tau})\|_\dagger^{[\rho]}\bigr]
=
\frac{1}{T}\Biggl(
\sum_{t=1}^{T-1}(1-\beta^{t})\,\E\bigl[\|\nabla L(\bar{\mW}_{t})\|_\dagger^{[\rho]}\bigr]
+ \frac{1-\beta^{T}}{1-\beta}\E\bigl[\|\nabla L(\bar{\mW}_{T})\|_\dagger^{[\rho]}\bigr]
\Biggr).
\]
Divide the previous inequality by $T$, substitute $D=\frac{1-\beta}{2\beta}\rho$,
and use $\E[L(\mW_0)-L(\mW_T)]\le L(\mW_0)-L(\mW^*)$ to obtain
\eqref{eq:o2nc_formal_bound}. This completes the proof.
\qed

\subsection{Proof of Theorem~\ref{thm_pion}}

For clarity, we assume the FTPL expectation is computed exactly; the finite-sample implementation follows similarly by a concentration argument.

\paragraph{Step 1: Bounding the discounted regret.}
By Theorem~\ref{thm:regret_FTPL}, for any $t\in[T]$,
\begin{equation}\label{eq:pion_regret_raw}
\Reg_t^{[\beta]}(1)
\le 2\sqrt{2}\,C\left(
\Tr\!\left[\sqrt{\sum_{s=1}^t \beta^{2(t-s)} \mG_s \mG_s^\top}\right]
+\frac{mG}{\beta}
\right),
\qquad
C\coloneq\Bigl(\frac{n}{n-m-1}\Bigr)^{1/4},
\end{equation}
where $G\coloneq \max_{s\in[T]}\|\mG_s\|_{\op}$.
Under Assumption~\ref{assm:gradient_bound}, we have
$\E[\mG_s\mG_s^\top]\preceq \mQ^2$.
Using Jensen's inequality and $\Tr(\sqrt{\mA})=\|\mA^{1/2}\|_*$ for $\mA\succeq 0$,
\begin{align}
\E\Tr\!\left[\sqrt{\sum_{s=1}^t \beta^{2(t-s)} \mG_s \mG_s^\top}\right]
&\le
\Tr\!\left[\sqrt{\sum_{s=1}^t \beta^{2(t-s)} \E[\mG_s \mG_s^\top]}\right] \nonumber\\
&\le
\Tr\!\left[\sqrt{\sum_{s=1}^t \beta^{2(t-s)} \mQ^2}\right]
=
\frac{\|\mQ\|_*}{\sqrt{1-\beta^2}}.
\label{eq:pion_regret_trace_bound}
\end{align}
Combining \eqref{eq:pion_regret_raw}--\eqref{eq:pion_regret_trace_bound} yields
\begin{equation}\label{eq:pion_regret_bound}
\E\bigl[\Reg_t^{[\beta]}(1)\bigr]
\le
2\sqrt{2}\,C\left(
\frac{\|\mQ\|_*}{\sqrt{1-\beta^2}}
+\frac{mG}{\beta}
\right).
\end{equation}

\paragraph{Step 2: Bounding the discounted noise.}
Under Assumption~\ref{assm:gradient_bound}, $\E[\mE_s\mE_s^\top]\preceq  \E[\mG_s\mG_s^\top]\preceq \mQ^2$.
By Jensen and the same trace-sqrt manipulation,
\begin{align}
\E\Bigl\|\sum_{s=1}^t \beta^{t-s}\mE_s\Bigr\|_*
&=
\E\Tr\!\left(\sqrt{\Bigl(\sum_{s=1}^t \beta^{t-s}\mE_s\Bigr)\Bigl(\sum_{s=1}^t \beta^{t-s}\mE_s\Bigr)^\top}\right) \nonumber\\
&\le
\Tr\!\left(
\sqrt{
\E\Bigl[\sum_{s=1}^t\sum_{r=1}^t \beta^{t-s}\beta^{t-r}\mE_s\mE_r^\top\Bigr]
}
\right)
\le
\Tr\!\left(\sqrt{\sum_{s=1}^t \beta^{2(t-s)} \E[\mE_s\mE_s^\top]}\right) \nonumber\\
&\le
\Tr\!\left(\sqrt{\sum_{s=1}^t \beta^{2(t-s)} \mQ^2}\right)
=
\frac{\|\mQ\|_*}{\sqrt{1-\beta^2}}.
\label{eq:pion_noise_bound}
\end{align}

\paragraph{Step 3: Plug into the O2NC bound and choose parameters.}
Proposition~\ref{prop:o2nc} (with $D= \frac{1-\beta}{2\beta}\rho$) together with
\eqref{eq:pion_regret_bound}--\eqref{eq:pion_noise_bound} implies
\begin{align}
\E_{\tau}\!\left[\,
\|\nabla L(\bar{\mW}_{\tau})\|_*^{[\rho]}
\right]
&\le \frac{G L}{DT}
+ \frac{1+(1-\beta)(T-1)}{T}\cdot
2\sqrt{2}\,C\left(
\frac{\|\mQ\|_*}{\sqrt{1-\beta^2}}
+\frac{mG}{\beta}
\right)\nonumber\\
&\quad
+ \frac{1+(1-\beta)(T-1)}{T}\cdot \frac{\|\mQ\|_*}{\sqrt{1-\beta^2}}.
\label{eq:pion_step3_start}
\end{align}
Using $1+(1-\beta)(T-1)\le \beta + (1-\beta)T$ and $\sqrt{1-\beta^2}\ge \sqrt{1-\beta}$, we obtain
\begin{align}
\E_{\tau}\!\left[\,
\|\nabla L(\bar{\mW}_{\tau})\|_*^{[\rho]}
\right]
&\le \frac{G L}{DT}
+ (2\sqrt{2}C+1)\Bigl(\frac{\beta}{T}+1-\beta\Bigr)\frac{\|\mQ\|_*}{\sqrt{1-\beta}}
\nonumber\\
&\quad
+ 2\sqrt{2}C\Bigl(\frac{\beta}{T}+1-\beta\Bigr)\frac{mG}{\beta}.
\label{eq:pion_step3_simplified}
\end{align}
Since $2\sqrt{2}+1\le 4$, we further simplify to
\begin{equation}\label{eq:pion_step3_clean}
\E_{\tau}\!\left[\,
\|\nabla L(\bar{\mW}_{\tau})\|_*^{[\rho]}
\right]
\le
\frac{G L}{DT}
+4C\|\mQ\|_*\Bigl(\frac{1}{T\sqrt{1-\beta}}+\sqrt{1-\beta}\Bigr)
+\frac{2\sqrt{2}CmG}{T}
+2\sqrt{2}C\frac{(1-\beta)mG}{\beta}.
\end{equation}

We now choose $\beta$ and $D$ as functions of $\varepsilon$.
Let
\[
\sqrt{1-\beta}\coloneq \frac{\varepsilon}{20C\|\mQ\|_*}
\qquad\Longleftrightarrow\qquad
\beta = 1-\Bigl(\frac{\varepsilon}{20C\|\mQ\|_*}\Bigr)^2,
\]
and assume $\varepsilon \le 10\sqrt{2}\,C\|\mQ\|_*$, so that $\beta\ge \tfrac12$.
Set $D=\frac{1-\beta}{2\beta}\rho$ as required by Proposition~\ref{prop:o2nc}.
Plugging these choices into \eqref{eq:pion_step3_clean} gives
\begin{align}
\E_{\tau}\!\left[\,
\|\nabla L(\bar{\mW}_{\tau})\|_*^{[\rho]}
\right]
&\le
\underbrace{\frac{2\beta}{(1-\beta)\rho}\cdot \frac{G L}{T}}_{\le \; \frac{800C^2(\|\mQ\|_*)^2G L}{\rho\,\varepsilon^2 T}}
+
\underbrace{\frac{4C\|\mQ\|_*}{T\sqrt{1-\beta}}}_{= \;\frac{80C^2(\|\mQ\|_*)^2}{\varepsilon T}}
+
\underbrace{4C\|\mQ\|_*\sqrt{1-\beta}}_{=\;\varepsilon/5}
\nonumber\\
&\quad
+\frac{2\sqrt{2}CmG}{T}
+ 2\sqrt{2}C\frac{(1-\beta)mG}{\beta}.
\label{eq:pion_final_preT}
\end{align}
The last term satisfies
\[
2\sqrt{2}C\frac{(1-\beta)mG}{\beta}
\le 4\sqrt{2}C(1-\beta)mG
=
4\sqrt{2}C mG \cdot \frac{\varepsilon^2}{400C^2(\|\mQ\|_*)^2}
=
\frac{\sqrt{2}\,mG\,\varepsilon^2}{100\,C(\|\mQ\|_*)^2},
\]
which is $\le \varepsilon/5$ provided $\varepsilon \le \frac{10\sqrt{2}\,C(\|\mQ\|_*)^2}{mG}$.
Under this additional small-accuracy condition, \eqref{eq:pion_final_preT} reduces to
\[
\E_{\tau}\!\left[\,
\|\nabla L(\bar{\mW}_{\tau})\|_*^{[\rho]}
\right]
\le
\frac{800C^2(\|\mQ\|_*)^2G L}{\rho\,\varepsilon^2 T}
+\frac{80C^2(\|\mQ\|_*)^2}{\varepsilon T}
+\frac{2\sqrt{2}CmG}{T}
+\frac{2\varepsilon}{5}.
\]
Hence, if
\[
T \;\ge\; \max\left\{
\frac{4000\,C^2(\|\mQ\|_*)^2G L}{\rho\,\varepsilon^3},
\;
\frac{400\,C^2(\|\mQ\|_*)^2}{\varepsilon^2},
\;
\frac{10\sqrt{2}\,C mG}{\varepsilon}
\right\},
\]
then $\E_{\tau}\!\left[\,
\|\nabla L(\bar{\mW}_{\tau})\|_*^{[\rho]}
\right]\le \varepsilon$, which means Pion outputs a $(\rho,\varepsilon)$-stationary point after $T$ iterations. This proves the theorem.

\subsection{Proof of Theorem~\ref{thm:Leon}}
The proof follows the same O2NC reduction as in Theorem~\ref{thm_pion}; the only change is the regret guarantee of the underlying online learner. Specifically, Leon (derived from FAML) achieves the same discounted-regret bound as Pion but without the dimensional factor, i.e., it corresponds to setting $C \;=\; \Bigl(\frac{n}{n-m-1}\Bigr)^{1/4}$ to $C=1$
in the proof of Theorem~\ref{thm_pion}. Substituting $C=1$ throughout yields the stated iteration complexity for Leon, with all other steps unchanged. For brevity, we do not repeat the argument.

\section{Empirical Validation: Additional Details}\label{app:experiments}

To empirically validate the stability--convergence behavior suggested by our theory, we compare \emph{Pion} and \emph{Leon} against \emph{Muon} on a synthetic Robust Matrix Sensing objective explicitly constructed to violate smooth-optimization assumptions. The purpose of this experiment is not to showcase best-case speed on benign losses, but to stress-test whether the \emph{intrinsic (implicit) smoothing} built into the Matrix OLO formulation of Pion and Leon translates into visibly steadier descent dynamics when gradients are discontinuous and the landscape contains oscillatory nonconvex structure.

The key distinction is that Pion and Leon admit explicit stability--convergence guarantees in our framework, whereas Muon does not. This gap becomes most evident in nonsmooth regimes: without an intrinsic smoothing mechanism, Muon can react sharply to abrupt changes in the gradient matrix---exactly the behavior induced by objectives with kinked terms or rapidly varying components. In contrast, Pion and Leon inherit an \emph{implicit smoothing effect} from the Matrix OLO geometry. Consequently, we expect Muon to exhibit more fluctuations on this stress test, while Pion and Leon follow smoother, more stable trajectories, with Leon typically enjoying a modest advantage consistent with our theoretical bounds.

We instantiate this stress test via a \emph{Robust Matrix Sensing with Nonconvex Ripples} objective:
\begin{equation*}
    f(\mX) = \frac{1}{m} \sum_{k=1}^m 
    \left(
        |\langle \mA_k, \mX \rangle|
        \cdot
        \left(1 - 0.9 \cos\!\left(3 \langle \mA_k, \mX \rangle\right)\right)
        + 0.5
    \right),
\end{equation*}
where $\mA_k$ are random measurement matrices with i.i.d.\ Gaussian entries. In our experiments, $\mX \in \reals^{d \times d}$ with $d=20$ and $m=100$ measurements. For Leon, we set $\beta_1=\beta_2=0.9$. The absolute-value term introduces $\ell_1$-type nonsmoothness and hence gradient discontinuities near $\langle \mA_k,\mX\rangle=0$, while the high-frequency cosine modulation creates persistent nonconvex ``ripples'' that repeatedly perturb the local geometry. Together, these components generate abruptly varying gradient signals designed to destabilize optimizers that track instantaneous gradients too closely.

\newpage

\printbibliography
        
\end{document}